\documentclass[11pt]{amsart}
\usepackage{amssymb,amsfonts,amsthm,amsmath}
\usepackage[usenames,dvipsnames]{color}
\usepackage{mathrsfs}
\usepackage[english]{babel}
\usepackage{graphicx}

 \newcommand{\obra}[3]{{\sc #1} {\em #2}. {#3}.}

\newtheorem{theorem-main}{\bf Theorem}
\newtheorem{theorem}{\bf Theorem}
 \newtheorem{lemma}[theorem]{\bf Lemma}
 \newtheorem{proposition}[theorem]{\bf Proposition}
 \newtheorem{definition}[theorem]{\bf Definition}
 \newtheorem{corollary}[theorem]{\bf Corollary}
 \newtheorem{remark}[theorem]{Remark}

 \renewenvironment{proof}{{\em \noindent Proof.}}
 {\hfill $\square$\newline}

 \newcommand{\R}{\mathbb{R}}
 \newcommand{\Q}{\mathbb{Q}}
 \newcommand{\N}{\mathbb{N}}

 \newcommand{\SSS}{\mathbb{S}}

 \newcommand{\EE}{\mathcal{E}}
 \newcommand{\FF}{\mathcal{F}}

 \newcommand{\CC}{\mathcal{C}}
  
   \newcommand{\RR}{\mathcal{R}}
   
    \newcommand{\LL}{\mathcal{L}}
\newcommand{\UU}{\mathcal{U}}
\newcommand{\VV}{\mathcal{V}}
\newcommand{\WW}{\mathcal{W}}
 \newcommand{\AAA}{\mathcal{A}}

\newcommand {\abs}[1]{\mid\! #1\!\mid}

\newcommand{\wht}[1]{{{\widehat{#1}}}}
\newcommand{\wt}[1]{{{\widetilde{#1}}}}

\newcommand {\absgamma}{{\abs{\gamma}}}

\newcommand{\bg}{{\bf g}}
\newcommand{\bh}{{\bf h}}

\newcommand{\bv}{{\bf v}}
\newcommand{\oth}{{\overline{\theta}}}
\newcommand{\clos}{{\rm clos}}
\newcommand{\crit}{{\rm crit}}
\newcommand{\rd }{{\rm d}}

\newcommand{\eps}{\varepsilon}
\newcommand{\g}{\gamma}

 \newcommand{\IN}{{\rm in}}

\newcommand{\nbh}{{\nabla_\bh}}

\newcommand{\oo}{{\bf 0}}

\newcommand{\sing}{{\rm sing}}

\newcommand{\vp}{\varphi}

\newcommand{\ww}{{\bf w}}
\newcommand{\xx}{{\bf x}}
\newcommand{\yy}{{\bf y}}

\title[On restricted analytic gradients on analytic I.S.S.]{On restricted Analytic Gradients
on Analytic Isolated Surface Singularities}
\author{Vincent Grandjean}
\address{{\it Permanent Address:} V. Grandjean, Department of Computer Science,
University of Bath, BATH BA2 7AY, England,(United Kingdom)} 
\email{cssvg@bath.ac.uk\vspace{-6pt}}
\address{{\it Temporary Address:} V. Grandjean, Fields Institute, 222 College Street,
Toronto, Ontario,  M5T 3J1, Canada}
\email{vgrandje@fields.utoronto.ca}
\author{Fernando Sanz}
\address{{\it Permanent Address:} F. Sanz, University of Valladolid,
Departamento de \'{A}lgebra, Geometr\'{\i}a y Topolog\'{\i}a, Facultad de
Ciencias, Prado de la Magdalena s/n, E-47006, Valladolid (Spain).
Fax: (34) 983 423788} \email{fsanz@agt.uva.es}
\begin{document}
\maketitle
\begin{abstract}
Let $(X,\oo)$ be a real analytic isolated surface singularity at 
the origin $\oo$ of a real analytic manifold $(\R^n,\oo)$ equipped 
with a real analytic metric $\bg$. Given a real analytic function 
$f_0:(\R^n,\oo) \to (\R,0)$ singular at $\oo$, we prove that the
gradient trajectories for the metric $\bg|_{X\setminus \oo}$ of the
restriction $(f_0|_X)$ escaping from or ending up at the origin
$\oo$ do not oscillate. Such a trajectory is thus a sub-pfaffian 
set. Moreover, in each connected component of $X\setminus \oo$ where 
the restricted gradient does not vanish, there is always a trajectory 
accumulating at $\oo$ and admitting a formal asymptotic expansion 
at $\oo$.
\end{abstract}

\section{Introduction}

Let $f_0:(\R^n,\oo)\to\R$ be a real analytic function such that
 $\oo$ is a critical point of $f_0$.
Let $\bg$ be a real analytic Riemannian metric defined in a
neighborhood of $\oo$. Let $\gamma:[0,+\infty[\to\R^n$ be a maximal
solution of the gradient vector field
$\nabla_\bg f_0$ such that $\omega(\gamma):=\lim_{t\to\infty}\gamma(t)=\oo$,
and let $\absgamma\subset\R^n$ be its image. We are
not interested in any particular parameterization and we will simply
call $\gamma$ and $\absgamma$ {\em a gradient trajectory}.
Gradient trajectories $\gamma:]-\infty,0]\to \R^n$ escaping from
$\oo=\lim_{t\to-\infty}\gamma(t)$
will be dealt with in same way in changing the sign of $f_0$.

The classical problem of the gradient is to know how, from an analytic point
of view, does the solution $\absgamma$ go to its limit point $\oo$.
For a long time remained undecided  Thom's question famously known as
{\it Thom's Gradient Conjecture}: does the trajectory have a {\it tangent}
at its limit point, namely does $\lim_{t\to\infty}\frac{\gamma(t)}{|\gamma(t)|}$ exist ?
(see \cite{Mou} for an historical account by then).
Eventually Kurdyka, Mostowski and Parusi\'nski showed that the length of the radial projection
of the curve $\absgamma$ onto $\SSS^{n-1}$ is finite \cite{KuMoPa}, thus proving Thom's Conjecture.

\medskip
A much more challenging question about the behavior of gradient trajectories at
their limit point is to decide whether they oscillate or not. A trajectory $\gamma$ is
(\textit{analytically}) {\it non-oscillating} if given any
(semi-)analytic subset $H \subset \R^n$ the intersection $\absgamma
\cap H$ has finitely many connected components.

The plane case is well understood. In dimension $n \geq 3$, but a few special 
cases in dimension $3$ \cite{Sa,FoSa,Go}, the non-oscillation of gradient 
trajectories is not known.

It is also worth recalling that in the case of real analytic vector fields on 
a $3$-manifold, some very interesting properties of the Hardy field of the real analytic 
function germs along a given non-oscillating trajectory have been studied in \cite{Can-M-R},
and thus allowing a partial reduction of the singularities result.

\smallskip
In the special case where a real analytic isolated surface singularity is foliated by
gradient trajectories, the main result of this paper guarantees,
that they do not oscillate. In fact, we will solve
the following slightly more complicated problem.

\medskip
Let $X\subset \R^n$ be a real analytic isolated surface singularity at the
origin $\oo$. Each connected component $S_0$ of the germ at $\oo$ of
$X\setminus\{\oo\}$ is a real analytic submanifold of $\R^n$. The ambient
metric $\bg$ induces on $S_0$ an analytic Riemannian metric $\bh :=\bg|_{S_0}$.
The gradient vector field $\nbh (f_0|_{S_0})$ of the restriction $f_0|_{S_0}$
of the function $f_0$ to $S_0$ is thus well defined. The vector field $\nbh (f_0|_{S_0})$
is called {\em the restricted gradient vector field} of $f_0$ on $S_0$ and will be shortened as
$\nbh (f_0)$.

The main result of this paper is the following:
\begin{theorem}\label{thm:MainResult}
Let $\gamma:\R_{\geq 0}\to S_0$ be a trajectory of the restricted
gradient vector field $\nbh f_0$ accumulating at $\oo$. Then
$\gamma$ is analytically non-oscillating.
\end{theorem}

A pleasant and cheap consequence of this result is
\begin{corollary} \label{cor:Pfaffian}
The curve $|\gamma|$ is a {\em sub-pfaffian set}.
\end{corollary}

A natural question is to ask whether there exists a trajectory
$\gamma$ of the restricted gradient accumulating to the origin to
apply the main theorem to. 
Elementary topological arguments and some properties of a gradient vector field show 
that it is always the case:  
\begin{proposition}\label{pro:trajectory-to-origin}
There exists a non-stationary trajectory of $\nabla_\bh f_0$
accumulating to $\oo$ either in positive or in negative time.
\end{proposition}

It is worth recalling that in the smooth context of an analytic
gradient vector field on $(\R^n,\oo)$, there exists furthermore a
real analytic curve through $\oo$ invariant for the gradient vector
field \cite{Mou}, {\em a real analytic separatrix}). For restricted
gradients over isolated surface singularities we also prove here
there always exists a {\em formal separatrix}:

\begin{theorem} \label{cor:formal-separatrix}
Let $S_0$ be a connected component of $X\setminus\{\oo\}$. If
$\nabla_\bh f_0$ does not vanish $S_0$, there exists a trajectory
$\gamma:\R_{\geq 0}\to S_0$ of $\nabla_\bh f_0$ accumulating to
$\oo$ which admits a formal asymptotic expansion at the origin such
that the associated formal curve $\widehat{\Gamma}$ is invariant for
the restricted gradient vector field.
\end{theorem}

\section{Structure of the proof}\label{Section:plan-article}

We first recall the case of an analytic Euclidean gradient in $\R^2$.

Trajectories of a real analytic vector field in $\R^2$ accumulating
at the origin either "spiral" around the origin or have a tangent.
In the latter case, a Rolle's type argument
shows that the trajectory is non-oscillating (see \cite{Can-M-S}).
The non-oscillation of a planar analytic gradient trajectory is
thus given by the existence of a tangent. Although Thom's Gradient conjecture
holds true (\cite{KuMoPa}), we sketch the usual simpler proof
of the existence of a tangent in the plane case. This will provide a
flavor of some of the arguments that makes our proof of Theorem
\ref{thm:MainResult} works.

\smallskip
\noindent Let $(r,\vp)$ be the polar coordinates at the origin of
$\R^2$ and write
$$f_0(r\cos \vp, r\sin \vp) = r^k [F_k (\vp) +
O(r)] $$
where $F_k(\vp)$ is the restriction to the unit circle of the
homogeneous part of $f_0$ of least degree. The gradient differential
equation becomes a differential equation on $\SSS^1 \times \R_{\geq
0}$ and, after division by $r^{k-1}$, writes as
\begin{equation}\label{Equation:Gradient-polar}
\dot{r} = r(kF_k +O(r)) \mbox{ and } \dot{\vp} = F_k^\prime + O(r).
\end{equation}
Since $F_k$ is not identically zero, when it is constant we divide
Equation (\ref{Equation:Gradient-polar}) by $r$ and find that
the divided vector field is transverse to $\CC=\SSS^1 \times 0$ at
each point (\textit{dicritical case}). If $F_k$ is not constant, then
$F_k^\prime$ must vanish and change sign along the circle $\CC$.
This prevents any gradient trajectory from accumulating on the whole
bottom circle $\CC$ (\textit{non-monodromic case}).
\\
In both cases, dicritical and non-monodromic, a plane gradient
trajectory does not spiral around its limit point, therefore it has
a tangent and thus does not oscillate.

\medskip
The plane case is enlightening enough to provide us with some of the
elements we need to prove Theorem \ref{thm:MainResult}. The surface
$S_0$ on which we want to understand the behavior of the restricted
gradient trajectories at their limit point $\oo$, is analytically
diffeomorphic to a cylinder $\SSS^1 \times ]0,\eps]$. We can carry
the metric $\bh$ over this cylinder, so that we have a well defined
gradient differential equation. Our concern then becomes: how does a
trajectory of this differential equation behave near the bottom
circle $\CC=\SSS^1 \times 0$? There is no canonical way to extend
the inverse of the diffeomorphism onto $\SSS^1 \times [0,\eps]$, and
so \`a-priori, our differential equation is not well defined on the
bottom circle $\CC$, if defined at any point of it!

Nevertheless, we manage to prove that a limit dynamics exists on the
circle $\CC$ but at finitely many points. We first show that our setting only allows a
single possible type of oscillation, that we call spiraling.
To keep up with the planar situation, we actually prove that
the only possible dynamics of the restricted gradient vector field will either
be dicritical-like or non-monodromic-like (see Section \ref{section:dynamic-on-cylinder}
for precise definitions). Consequently, trajectories cannot spiral and will
therefore be non-oscillating.

\vspace{6pt} \noindent
In Section \ref{section:ParameterizationSurfaces},
Proposition \ref{pro:parameterization-of-surfaces} provides a systematic
way to parameterize $\clos(S_0)$ as the surjective image of
a continuous mapping defined on $\SSS^1\times[0,\eps]$ which induces an analytic
diffeomorphism between the open cylinder $\SSS^1\times]0,\eps]$ and $S_0$. Such a
parameterization is inherited from the resolution of singularities of the analytic
surface $X$ and thus comes with some very specific properties on the bottom circle $\CC
=\SSS^1\times 0$.
\\
In Section \ref{section:expansion-function}, we use such a
parameterization to express the pull-back of the
restriction of the function $f_0$ to $S_0$, as well as the
corresponding gradient vector field, in polar-like coordinates
$(\vp,r)\in \SSS^1 \times [0,\epsilon]$ as in Equation
(\ref{Equation:Gradient-polar}). We obtain a continuous
principal part along the bottom circle that will play a similar role
to that of the principal part $F_k$ in
(\ref{Equation:Gradient-polar}).
\\
Section \ref{section:dynamic-on-cylinder} deals with the oscillating
dynamics of a given real analytic vector field on an
isolated surface singularity (such as $S_0$) and vanishing at the
tip, which can only be spiraling around this singular point, as we
have already suggested. Although of an independent nature, we use
the results of the previous sections for the proof. We also
describe two local dynamical situations we call ``dicritical'' and
``non-monodromic'', generalizing the planar smooth case,
and show here that such dynamics are non-oscillating.
\\
Our notion of ``dicritical-ness": {\it there
exists an arc of the bottom circle $\CC$ such that each point is the
$\omega$-limit point of a unique trajectory}, is
weaker than the usual notion requiring transversality to the
exceptional divisor (here the bottom circle). Our notion
of ``non-monodromic-ness'' is also weaker than the notion stated above: the function
playing the role of $F_k$ in Equation (\ref{Equation:Gradient-polar}), is
continuous, not constant but can fail to be differentiable at finitely many
points of $\CC$.
\\
The proof of Theorem \ref{thm:MainResult} is done in Section
\ref{section:ProofMainResult}. It uses all the main results of
Sections \ref{section:ParameterizationSurfaces},
\ref{section:expansion-function}
to obtain a differential equation on a cylinder $\SSS^1 \times
[0,\eps]$ which is analytic on $\SSS^1 \times ]0,\eps]$. Although there
is a slight cost, namely a finite subset of the bottom circle where the
differential equation is likely to be not defined, we know enough
about it to show that only the dicritical or non-monodromic situations happen.
Section \ref{section:dynamic-on-cylinder} then guarantees the non-oscillation of
the restricted gradient trajectories.
\\
The last section deals with two not-so-unexpected consequences of
our main result, Corollary~\ref{cor:Pfaffian} and
Theorem~\ref{cor:formal-separatrix}.
\section{Parameterization of real analytic
surfaces}\label{section:ParameterizationSurfaces}
Let $X$ be the germ, at the origin $\oo$ of $\R^n$, of a real analytic surface of pure dimension $2$.
We will not distinguish between the germ of $X$ at
$\oo$ and a representative in a sufficiently small neighborhood of $\oo$.\\
Assume that the surface $X$ has an isolated singularity at the origin,
that is $X\setminus\{\oo\}$ is a smooth embedded analytic surface of $\R^n$.

\medskip
Let $S_0$ be a given connected component of the germ at $\oo$ of the
regular part $X\setminus\{\oo\}$. The \textit{tangent cone of $S_0$
at $\oo\in\R^n$} is the subset of $\SSS^{n-1}$ made of the limits of
the oriented secant direction $\frac{p_k}{|p_k|}$
taken along sequences of points $(p_k)_k$ in $S_0$ converging to
$\oo$. The tangent cone $C_\oo(S_0)$ is a compact connected
subanalytic subset of $\SSS^{n-1}$ of dimension at most one. We
distinguish two cases:\\
- If $C_\oo(S_0)$ reduces to a single point, we will speak about the
\textit{cuspidal tangent cone} case (CTC for short).
\\
- If $C_\oo(S_0)$ is a curve we will speak of the \textit{open
tangent cone} case (OTC).

\medskip
For any $\varepsilon>0$ sufficiently small, the Local Conic
Structure Theorem (see \cite{Mil,Boc-C-R,vdD}) states that $X$ is
homeomorphic to the cone with vertex $\oo$ over
$X_\varepsilon=X\cap\SSS^{n-1}_\varepsilon$, where
$\SSS^{n-1}_\varepsilon$ is the Euclidean sphere of radius $\varepsilon$.
Moreover, the surface $X$ is transverse to $\SSS^{n-1}_{\eps}$
so that $S_0\cap \SSS^{n-1}_{\varepsilon}$ is analytically
diffeomorphic to $\SSS^{1}$ and $S_0\cap\clos(B(\oo,\eps))$ is
analytically diffeomorphic to $\SSS^{1}\times]0,\eps]$.

\begin{definition}
Assume  $C_\oo (S_0)$ consists of the single oriented direction $\eta \in \SSS^{n-1}$.
A system of analytic coordinates $(\xx,z)=(x_1,\ldots,x_{n-1},z)$ at $\oo$ is called \em
adapted for $S_0$ \em if the half-line $\R_+\eta$ is the non-negative $z$-axis.
\end{definition}

Given adapted coordinates $(\xx,z)$ in the CTC case, taking the
height function $z$ instead of the distance function, the proof of
the Local Conic Structure's Theorem adapts to obtain the same
conclusion: the intersection $S_0\cap \{z=\varepsilon\}$ is transverse,
thus analytically diffeomorphic to $\mathbb{S}^{1}$
and $S_{0}\cap\{0<z\leq\eps\}$ is analytically diffeomorphic to
$\mathbb{S}^{1}\times]0,\eps]$ for $0<\eps\leq\eps_0$ once $\eps_0$ is
sufficiently small.

From now on, we fix some $\eps_0$ so that in both cases OTC or CTC,
the above properties coming from the locally conic structure are
satisfied. We consider a representative of $S_0$ in $\{0<z<\eps_0\}$,
where $z$ stands for the distance to the origin in the OTC case and
for the last component of an adapted system of coordinates in the CTC case.

\medskip
In what follows we will desingularize the surface $S_0$. First, it will
be convenient for us to \em open \em the surface $S_0$ by means of a single
blowing-up-like mapping $\beta$. Roughly speaking, we mean that the
inverse image of $S_0$ by $\beta$ accumulates to a one-dimensional set
in the exceptional divisor.\\
In the OTC case, the usual polar blowing-up $\beta:({\yy},r)\mapsto
r\yy$, for $\yy\in \SSS^{n-1}$ and $r$ the distance function,
``opens'' the surface $S_0$, since $\beta^{-1}(S_0)$ accumulates
onto $C_\oo (S_0)\subset\SSS^{n-1}$, a subanalytic curve.
\\
The CTC case, however, requires more work. Starting with an adapted
system of coordinates $(\xx,z)$, a first and naive candidate mapping
to ``open" the surface is a ``ramified blowing-up" of the form
$\beta_s:(\yy,w)\mapsto (w^s\yy,w)$, where $\yy\in\R^{n-1}$, for
a well chosen rational exponent $s>1$. Such an
exponent $s$ exists when the $z$-axis is contained in
the surface $S_0$. However the surface $\beta_s^{-1}(S_0)$ may accumulate
to a single point on the divisor $\beta_s^{-1}(z=0)$ (or escapes to infinity)
whatever the exponent $s$ is. In such a case the surface $S_0$ cannot
be opened with any such ramified blowing-up. In this situation, we consider
a given analytic half-branch on $S_0$ as new non-negative $z$-axis, and
in these new coordinates, a ramified blowing-up as above will open the
surface.
\\
The next technical lemma will detail such considerations.
First, an {\em analytic half-branch at the origin $\oo$ of $\R^n$} is the
germ at $\oo$ of a connected component $\Gamma$ of $Y\setminus\{\oo\}$,
where $Y$ is a one-dimensional analytic set through $\oo$. When
$\Gamma$ is contained in $\{z>0\}$, it is parametrized as the image of
an analytic mapping $z\mapsto(\theta(z),z^N),z>0$, where
$\theta=(\theta_1,\ldots,\theta_{n-1}):]-\varepsilon,
\varepsilon[\to\R^{n-1}$ is analytic with $\theta(0)=\oo$
and $N$ is a positive integer.

\begin{lemma}\label{lm:the-exponent-e}
Assume the tangent cone $C_\oo (S_0)$ is reduced to a point.
Let  $(\xx,z)$ be adapted analytic coordinates at $\oo$. \\
(i) There is a unique rational number $\nu>1$ such that the
accumulation set of the mapping $S_0 \ni (\xx,z) \to
\frac{\xx}{z^\nu} \in \R^{n-1}$
is a bounded subset of $\R^{n-1}$ and contains a point that is not $(0,\ldots,0)$.\\
(ii) There exists a unique positive rational number $e\geq \nu$ such
that the accumulation set of the mapping $S_0 \times S_0 \ni
((\xx,z),(\yy,z)) \to \frac{|\xx-\yy|}{z^e} \in \R$
is a bounded subset of
$\R$ containing a positive number. \\
(iii) Let $\Gamma:z \to (\theta (z),z^N)$ be a real analytic
half-branch at $\oo$ such that $\Gamma\subset S_0$. Then, the set of
accumulation values of the mapping $\tau_{e,\Gamma}:S_0\to\R^{n-1}$,
$(\xx,z) \mapsto \frac{\xx-\theta (z^{1/N})}{z^e}$ is a connected
bounded subanalytic set of dimension $1$.
\end{lemma}
\begin{proof}
The uniqueness of $\nu$ and $e$ are clear. \\
For (i), let $h(z):=\sup\{|\xx| \mbox{ for } (\xx,z)\in S_0\}$.
The function $h$ is subanalytic and extends continuously to $z=0$ by
$h(0)=0$. Writing it as a Puiseux's series $h(z)=az^\nu+\cdots$ with
$a\neq 0$, the exponent $\nu$ satisfies the required properties: the
cuspidal nature of $S_0$ and the definition of adapted coordinates
imply that $\nu>1$.

\smallskip
\noindent We show the existence of $e$ of point (ii) similarly to point
(i): We take this time the function $h$ to be defined as
$h(z):=\sup\{|\xx-\yy| \mbox{ for } (\xx,z),(\yy,z)\in S_0\}$.

\smallskip
\noindent For (iii), let $\Lambda$ be the set of accumulation values of
the mapping $\tau=\tau_{e,\Gamma}$. Since $\Gamma$ is contained in $S_0$,
the origin $\oo$ of $\R^{n-1}$ is in $\Lambda$. By definition of the exponent
$e$ of point (ii), $\Lambda$ is bounded and contains a point $p\neq\oo$.
The connectedness and subanalyticity of $\Lambda$ follow from the connectedness
of $S_0$ and the subanalyticity of $\tau$.
\end{proof}

\begin{remark}\em
The numbers $\nu, e$ of Lemma \ref{lm:the-exponent-e} depend on the
adapted system of coordinates. Take in $\R^3$ the revolution surface
$x^2 + y^2 - z^5 = 0$, then $e=\nu=5/2$. Consider now the change of
coordinates $(x',y',z') = (x+z^2,y,z)$, then $e'=\nu'= 2$. \em
\end{remark}

The next result synthesizes the discussion about the
possible \em opening  \em of the surface $S_0$ by a single
blowing-up-like mapping. Its proof follows from Lemma
\ref{lm:the-exponent-e}.
\begin{proposition}\label{pro:opening-blow-up}
In the OTC case, let $M=\SSS^{n-1}\subset\R^n$ with
coordinates $\yy=(y_1,\ldots,y_n)$. In the CTC case, let
$M = \R^{n-1}$ with coordinates $\yy=(y_1,\ldots,y_{n-1})$. Let
$(\xx,z)$ be  adapted coordinates for $S_0$ at
$\oo$ and let $e\in\Q_{>1}$ be the exponent of point (ii) in
Lemma~\ref{lm:the-exponent-e} for these adapted coordinates and
let $z\mapsto(\theta(z),z^N)$ be a parametrization of an
analytic half-branch $\Gamma$ in $S_0$ such that $eN \in \N$.
Consider the following analytic mapping
\begin{equation}\label{eq:map-beta}
\begin{array}{rccl}
\beta:  & M \times [0,\eps_0] & \to & \R^{n}
\\
 & (\yy,z) &  \mapsto & \left\{
 \begin{array}{ll}
 z\yy, & \hbox{OTC case,} \\
(z^{eN}\yy + \theta (z),z^N), & \hbox{CTC case}
 \end{array}
 \right.
\end{array}
\end{equation}
Then $\beta$ induces a diffeomorphism from $M\times]0,\eps_0]$ onto
its image. Let $S:=\beta^{-1}(S_0)$, $D:=\{z=0\}\subset
M\times\R$ and $E:=\clos(S)\cap D$. Then $E$ is a closed bounded
connected subanalytic curve of $D$ of dimension one.
\end{proposition}

A mapping $\beta$ as in (\ref{eq:map-beta}) is called an {\em opening blow-up of
$\clos (S_0)$}. In the CTC case, $\beta$ depends on the adapted
system of coordinates, on the given curve $\Gamma$ on $S_0$ and
on the number $N$ in the parametrization of $\Gamma$. As we
will see, the choice of all these parameters will not matter for our purpose, so
we do not need the notation $\beta$ to carry these parameters.

\bigskip
For the rest of this section, assume that we have picked an opening
blow-up $\beta$ of the surface $S_0$. A key element in our result
relies on the construction of an explicit diffeomorphism between $S$
and the open cylinder $\SSS^1\times]0,\eps_0]$, which extends to a
global parameterization of $\clos(S)=S\cup E$: a surjective
continuous mapping $\Phi: \SSS^1 \times [0,\eps_0] \to \clos (S)$. For
this purpose, we first resolve the singularities of the surface
$\clos(S)$, also providing a resolution of the
singularities of $\clos (S_0)$ (up to ramification). Several
formulations are possible. The version we use is stated in the
following theorem, an avatar of the general theory on reduction of
singularities of real analytic space as found in Hironaka \& Al.
\cite{Hir,Aro-H-V} (see also \cite{Bie-Mil2}).
\begin{theorem}[Reduction of singularities of $S$]\label{th:reduction-singularities}
There exists a non-singular real analytic surface $\wt{S}$, a
normal crossing divisor $\wt{E}\subset\wt{S}$ and a
proper analytic mapping $\sigma:\wt{S}\to \UU$ where $\UU$ is an open
neighborhood $\UU$ of $E$ in $M\times\R$ such that:
\begin{itemize}
\item[(i)] $\clos(S)\cap \UU\subset\sigma(\wt{S})$ and $ \sigma^{-1}(E) \subset\wt{E}$,
\item[(ii)] $S'=\sigma^{-1}(S)$ is an open submanifold of $\wt{S}$
 and the restricted mapping $\sigma|_{S'}:S'\to  S\cap \UU$ is an isomorphism,
\item[(iii)] If $E'=\clos (S')\cap\wt{E}$, then $E'=\clos(S')\setminus S'$, it
 is a compact subanalytic connected curve
of $\wt{S}$ and $\sigma(E')=E$.
\item[(iv)] If $p\in E'$, there is a fundamental system of neighborhoods $\{\WW_k\}$
of $p$ in $\wt{S}$ such that any connected component of $\WW_k\setminus\wt{E}$ is either contained in
$S'$ or has empty intersection with $S'$.
\end{itemize}
\end{theorem}
\begin{proof}
Let $X_1=\clos(\beta^{-1}(X\setminus\{\oo\}))$ be the strict
transform of $X$ by the opening blowing-up $\beta$ and let
$Z=(X_1\cup D)\cap \UU$ on some open neighborhood $\UU$ of $E$ in
$M\times\R$. The general reduction of singularities applied to the real closed
analytic set $Z$ states there exists a proper surjective analytic mapping
$\pi:\widetilde{M}\to \UU$, composition of finitely many blowing-ups
with closed analytic smooth centers, such that the total transform
$\pi^{-1}(Z)$ has only normal crossings. Moreover, the smooth
centers of blowing-ups are chosen either to be contained in the
singular locus of the corresponding strict transform of $Z$ or in
the divisors created along the resolution process. Since $\sing
(Z)\cap \clos(S)\subset D$, the mapping $\pi$ induces an isomorphism
from $\pi^{-1}(\UU\setminus D)$ onto $\UU\setminus D$. Let
$\widetilde{S}$ be the irreducible component of $\pi^{-1}(Z)$
containing $\pi^{-1}(S)$. Let
$\widetilde{E}=\pi^{-1}(D)\cap\widetilde{S}$ and put
$\sigma=\pi|_{\widetilde{S}}$. Since $\widetilde{S}$ is closed in
$\widetilde{M}$ and $\sigma$ is proper, we obtain the first inclusion
in point (i). The second inclusion is given by construction.
Since $S\cap \UU\subset \UU\setminus D$ and $\pi$ is an isomorphism
on $\pi^{-1}(\UU\setminus D)$ we get point (ii).
To prove point (iii), we first remark that $E'=\clos(S')\setminus
S'$ as an easy consequence of (i). The properness of $\sigma$
ensures that $E'$ is the Hausdorff limit as $\eps \to 0$ of the
subanalytic family of compact sets
$\wt{\CC}_{\eps}=\sigma^{-1}(S\cap\{z=\eps\})$, each analytically
diffeomorphic to the circle, and so $E'$ is subanalytic, compact and
connected. It cannot be reduced to a single point $p$ since,
otherwise the curve selection lemma would show that
$S'\cup\{p\}\subset\wt{S}$ is locally open at $p$ and thus $p$ would
be isolated in $\wt{E}$ which cannot be. The properness of $\sigma$
is used again to prove that $\sigma(E')=E$.
Finally, for point (iv), let $\WW$ be an affine chart
at $p$, isomorphic to $\R^2$, such that $\WW\cap\wt{E}$ is either one
or the two coordinate axis. Let $\WW_k=[-1/k,1/k]^2$. A connected
component of $\WW_k\setminus\wt{E}$ is either a half-space or a
quadrant. Each contains a single connected component of
$\WW_{k+1}\setminus\wt{E}$. If the property described in point (iv)
does not hold, there will be points in $\WW_k\setminus\widetilde{E}$
which belong to the boundary $E'=\clos(S')\setminus S'$ of $S'$,
thus impossible since $E'\subset\widetilde{E}$.
\end{proof}

A triple $\RR=(\wt{S},\wt{E},\sigma)$ satisfying the properties
(i-iv) of Theorem~\ref{th:reduction-singularities} will be called a
\textit{(total) resolution of singularities of $S$}. The curve
$\widetilde{E}$ will simply be called the \textit{divisor of the
resolution $\RR$}. The surface $S'=\sigma^{-1}(S)$ and $E'=\clos(S')\cap\wt{E}$ will be respectively
called the {\em strict transform} and the {\em strict divisor of the resolution}.
We will also speak of
$\mathcal{R}'=(S',E',\sigma'=\sigma|_{S'})$ as the {\em strict resolution of S
(associated to $\mathcal{R}$)}.

Let $\mathcal{R}=(\widetilde{S},\widetilde{E},\sigma)$ be  a
resolution and $p$ be a point of $\widetilde{E}$. Let
$\sigma_p:\wt{S}^1\to\wt{S}$ be the blowing-up of $\wt{S}$ at $p$.
This provides a new triple $\RR_p=(\wt{S}^1,\sigma_{p}^{-1}(\wt{E}),
\sigma\circ\sigma_p)$ which is a new resolution of singularities of
$S$.
\begin{definition}\label{def:resolution-dominates}
Let $\RR^1,\RR^2$ be two resolutions of the surface $S =
\beta^{-1}(S_0)$. If $\RR^2$ is obtained from $\RR^1$ by finitely
many successive points blowing-ups at points in the successive
corresponding divisors, we will say that \textit{$\RR^2$ dominates
$\RR^1$} and will write $\RR^2\succeq\RR^1$.
\end{definition}
A resolution dominating a given one will be obtained when we want to ``monomialize" one
or several functions on $S$ which are restrictions of analytic functions.
\begin{definition}\label{def:(S,H)-resolution}
Let $H=(h_1,\ldots,h_k)$ be a $k$-uple of real analytic functions in
a neighborhood of $E$ in $M\times \R_{\geq 0}$. A resolution
$\RR=(\widetilde{S},\widetilde{E},\sigma)$ of $S$ is {\em adapted to
$H$} (or briefly a {\em $(S,H)$-resolution}) if, for any $j$, the
composition $\widetilde{h}_j=h_j\circ\sigma$ has a {\em monomial
representation} at any point $p\in\widetilde{S}$: There are analytic
coordinates $(u,v)$ of $\widetilde{S}$ at $p$ such that
$\widetilde{h}_j=u^av^bG_j(u,v)$, where $a,b\in\N$, $G_j$ is
analytic and $G_j(0,0)\neq 0$.
\end{definition}
\begin{corollary}
Let $H=(h_1,\ldots,h_k)$ be as above and suppose that the restriction $h_j|_S$
has no critical point. Then there exists a resolution $\RR$ of $S$ such that, for any
$\RR^1\succeq\RR$, $\RR^1$ is a $(S,H)$-resolution.
\end{corollary}
\begin{proof}
From classical results in local monomialization of analytic
functions in a smooth analytic manifolds (see for instance
\cite{Bie-Mil1}): just consider a resolution of $S$ and blow-up
the points of the divisor where the corresponding total
transform of the $h_j$ have not yet a monomial representation.
\end{proof}

The following terminology is needed to state the principal
result of this section. Let $N$ be a real analytic manifold with
real analytic smooth boundary $\partial N$ and $f:N\to\R$ be a
continuous map. The function $f$ is {\em ramified-analytic at a point
$p$ of $\partial N$}, if there are $l\in\N$ and analytic
coordinates $(\xx,z)$ at $p$ for which $N=\{z\geq 0\}$ and $\partial
N=\{z=0\}$, such that the mapping $(\xx,z)\mapsto f(\xx,z^{l})$ is
analytic at $({\bf 0},0)$. If $h:N\to M$ is a continuous mapping
into an analytic manifold $M$, the mapping $h$ will be called {\em
ramified-analytic at $p\in\partial N$} if, in some analytic coordinates of $M$,
its components are ramified-analytic at $p$.

\begin{remark}\label{rk:derivative-wrt-r}{\em
Let $f:N\to\R$ be a ramified-analytic function at some point
$p\in\partial N$, with analytic coordinates $(\xx,z)$ at $p$ for
which $N=\{z\geq 0\}$ and $\partial N=\{z=0\}$. The function
$z\partial_z f$ extends continuously, in a
neighborhood $\VV$ of $p$, into a function which is ramified-analytic
at $p$ and, moreover, vanishes along the boundary $\VV \cap
\partial N$.}
\end{remark}
\begin{proposition}\label{pro:parameterization-of-surfaces}
Let $\RR=(\wt{S},\wt{E},\sigma)$ be a $(S,z)$-resolution and $\RR'=(S',E',\sigma')$ be the associated
strict resolution. There exist $\eps>0$ and a continuous mapping
$\wt{\Phi}:\SSS^1\times[0,\eps]\to\wt{S} $ with the following properties:
\begin{itemize}
\item[(i)] It maps $\SSS^1\times\{r\}$ onto  $\sigma^{-1}(S\cap\{z=r\})$ for
$0<r\leq\eps$ and induces an analytic diffeomorphism between $\SSS^1\times]0,\eps]$
and $\sigma^{-1}(S\cap\{0<z\leq\eps\})$.
\item[(ii)] It maps surjectively $\SSS^1\times[0,\eps]$ onto $\clos(S')=S'\cup E'$ and it maps
$\CC=\SSS^1\times\{0\}$ onto $E'$.
\item[(iii)] The set $\Omega= \Omega(\wt{\Phi})=(\wt{\Phi})^{-1}(E'\cap {\sing\,}(\wt{E}))
\subset \CC$ is finite and $\wt{\Phi}$ is \em uniformly
ramified-analytic \em at any point of $\CC\setminus\Omega$: there
exists $l\in\N$ such that $(\vp,r)\mapsto\wt{\Phi}(\vp,r^l)$ is
analytic at every point of $\CC\setminus\Omega$.
\end{itemize}
\end{proposition}

\medskip\noindent
Using Theorem \ref{th:reduction-singularities}, points (i), (ii) and
(iii) are true for $\Phi:=\sigma\circ\widetilde{\Phi}$ when replacing the strict
transforms $S'$ and $E'$ with the initial subsets $S$and $E$ respectively. Namely,
$\Phi$ maps surjectively $\SSS^1\times[0,\eps]$ onto $\clos(S)=S\cup E$, $\CC$ onto $E$
and $\SSS^1 \times ]0,\eps]$ diffeomorphically onto $S$, sending
$\SSS^1\times\{r\}$ onto $S\cap\{z=r\}$. Moreover, $\Phi$ is
uniformly ramified analytic at every point of $\CC\setminus\Omega$.
A mapping $\wt{\Phi}$ (or $\Phi$) satisfying points (i) to (iii)
of Proposition \ref{pro:parameterization-of-surfaces} is called a {\em parameterization
associated to the resolution $\RR$},  and the subset $\Omega$ in (iii) and is called
the {\em exceptional set of the parameterization $\wt{\Phi}$ (or $\Phi$)}.

\smallskip\noindent
\begin{proof}
Let $\RR=(\wt{S},\wt{E},\sigma)$ be a $(S,z)$-resolution. We construct a retraction of a
neighborhood of $\wt{E}$ in $\wt{S}$ onto $\wt{E}$ by integration of a certain
analytic vector field. It is just an avatar of the construction of a {\em Clemens
structure} on an analytic manifold equipped with a normal crossings divisor
(see \cite{Cle,Roc}).

\smallskip\noindent
Let $\wt{g}$ be an analytic Riemannian metric on $\wt{S}$,
whose existence is guaranteed by Grauert's Theorem on the analytic
embedding of analytic manifolds in Euclidean spaces \cite{Gra}.
Let $\tilde{z}:=z\circ\sigma:\wt{S} \to \R$. Let $\xi=\nabla_{\wt{g}}(-\tilde{z}^2)$
be the gradient vector field of $-\tilde{z}^2$ w.r.t the metric $\wt{g}$. Its
singular set is exactly the divisor $\wt{E}=\{\wt{z}=0\}$.

\smallskip\noindent
Let  $\eps$ be small enough so that $\sigma$ induces a
diffeomorphism from $\sigma^{-1}(S\cap\{0<z\leq\eps\})$ to
$S\cap\{0<z\leq\eps\}$. We can now consider $S$ just as being
$S\cap\{0<z\leq\eps\}$.

\smallskip\noindent
For $r\in ]0,\eps]$, let $\wt{\CC}_r=\tilde{z}^{-1}(r)=\sigma^{-1}(S\cap\{z=r\})$.
It is an embedded curve in $\wt{S}$ isomorphic to the circle $\SSS^1$. Let
$\rho:\SSS^1\to\wt{\CC}_\eps,\;\;\;\vp \mapsto\rho(\vp)$ be an
analytic diffeomorphism. For $p\in S'$, let $\gamma_p$ be the
maximal integral curve of $\xi$ with initial data $\gamma_p(0)=p$.
The parameterized curve  $\gamma_p$ is defined for times $t\geq 0$ and stays in $S'$.
Since the function $t\mapsto \tilde{z}(\gamma_p(t))$ strictly decreases to $0$ as $t$
goes to infinity
$\gamma_p$ cuts (orthogonally) each curve $\wt{\CC}_r$ for $r\in ]0, \tilde{z}(p)]$ only once.
Thanks to \L ojasiewicz's Gradient Inequality \cite{Loj}, the omega-limit set $\omega(\gamma_p)$
consists of a single point $R(p) \in E'$ and the mapping $R:\wt{S}\to E'$ is continuous
since $\wt{E}$ is compact. The following mapping is thus well defined:
\begin{equation}\label{eq:tildePsi}
\wt{\Phi}:\SSS^1\times[0,\eps]\to \wt{S},\;\;\;
\wt{\Phi}(\vp,r)= \left\{
\begin{array}{ll}
\vspace{4pt}
\wt{\CC}_r\cap\abs{\gamma_{\rho(\vp)}}, & \hbox{if $r\neq 0$;} \\
 R(\rho(\vp)), & \hbox{if $r=0$,}
\end{array}
\right. \end{equation}
where $\abs{\gamma_p}\subset\wt{S}$ is the image set of $\gamma_p$.
The restriction of $\wt{\Phi}$ to the open cylinder
$\SSS^1\times]0,\eps]$ is an analytic diffeomorphism onto $S'$, proving point (i).

\smallskip
\noindent
In order to obtain the continuity of $\wt{\Phi}$ and
properties (ii) and (iii), we will show that for $p\in E'$ there
exists $\vp_0\in\mathbb{S}^1$ such that
$\widetilde{\Phi}(\vp_0,0)=p$, $\widetilde{\Phi}$ is continuous at
$(\vp_0,0)$ and ramified-analytic if $p\in E'\setminus \sing(\widetilde{E})$.

\smallskip
\noindent
Let $p\in E'\setminus\sing\wt{E}$.
Let $(u,v)$ be analytic coordinates  at $p$ such that $\tilde{z}(u,v)= v^m$ with
$m\geq 1$ and $\wt{E}=\{v=0\}$. From point (iv) of Theorem~\ref{th:reduction-singularities},
there is a neighborhood $\VV$ of $p$ such that the half-space $\{v>0\}$
is contained in $S^\prime$.
The metric writes $\wt{g}=A\rd u^2+2B\rd u \rd v+C\rd v^2$, and we obtain
\begin{center}
\vspace{4pt}
$
\xi=2({\rm det\,}\wt{g})^{-1}(Bmv^{2m-1}\frac{\partial}{\partial
u}-Amv^{2m-1}\frac{\partial}{\partial v}).
$
\vspace{4pt}
\end{center}
Since $A(p)\neq 0$, the divided vector field $\xi' :=v^{1-2m}\xi$ is
not singular, transverse to the divisor $\wt{E}$ at $p$ and
generates the same foliation as $\xi$ on $\{v\neq 0\}$.
Thus there exists a trajectory
$\absgamma$ of $\xi$ with $\omega(\gamma)=p$ which extends smoothly
and analytically through of $p$ as a trajectory $|\gamma'|$ of
$\xi'$. Going backwards in time, $\absgamma$ cuts
$\wt{C}_\eps$ at a point $\rho(\vp_0)$ for some $\vp_0\in \SSS^1$.
Thus $p=R(\rho(\vp_0))= \widetilde{\Phi}(\vp_0,0)$. Let $\gamma_q'$
be the trajectory of $\xi'$ through a point $q\in \VV$. Since $\xi'$
is not singular in $\VV$ and transverse to the fibers $v=cst$, up
to shrinking $\VV$, the following mapping
\begin{center}
\vspace{4pt} $H:\VV\times]-\delta,\delta[\to\wt{S}$, $(q,t) \mapsto
H(q,t) := v^{-1}(t) \cap |\gamma_q'|$, \vspace{4pt}
\end{center}
is analytic. Fix $v_0>0$ such that $\gamma$ cuts $v^{-1}(v_0)$
inside $\VV$ and denote $\psi:\SSS^1\to\wt{C}_{v_0^{1/m}}$,
$\psi(\vp)=\wt{\Phi}(\vp,v_0^{1/m})$, an analytic diffeomorphism. By
construction the mapping we are looking for satisfies
\begin{center}
\vspace{4pt}$
\wt{\Phi}(\vp,r)=H(\psi(\vp),r^{1/m})
$\vspace{4pt}
\end{center}
in some neighborhood of $(\vp_0,0)$ and thus is ramified analytic at that point.\\
The number $m$ can be chosen constant for each connected component
of $E\setminus \sing(\wt{E})$, which are finitely many.
Thus there is a uniform ramification index $l$ along $\CC\setminus\Omega$.
So we get (iii).

\smallskip
\noindent Let $p\in E'\cap {\sing}\,\wt{E}$. Let $(u,v)$ be analytic
coordinates at $p$ such that $\tilde{z}(u,v)= u^lv^m$ with $l,m\geq
1$ and $\wt{E}=\{uv=0\}$. From point (iv) of Theorem
\ref{th:reduction-singularities} we assume that the first quadrant
$Q=\{u>0,v>0\}$ is contained in $S'$ for $u,v$ small enough. The
metric writes as $\wt{g}=A\rd u^2+2B\rd u\rd v+C\rd v^2$, and we
obtain
\begin{center}
\vspace{4pt}
$
\xi=2({\rm
det\,}\wt{g})^{-1}u^{2l-1}v^{2m-1}[(-lCv+mBu)\frac{\partial}{\partial
u}+(lBv-mAu)\frac{\partial}{\partial v}].
$
\vspace{4pt}
\end{center}
Since $\wt{g}$ is positive definite, the divided vector field
$\xi'=u^{1-2l}v^{1-2m}\xi$ has a saddle-type singularity
at $p$: its linear part $L_p$ at $p$ has two non-zero eigen-values
with opposite sign. Moreover, each eigen-direction is transverse to the
$u$-axis and $v$-axis, namely the components of $\wt{E}$ at $p$.
The only trajectories of $\xi'$ with $\omega$-limit point $p$
are the two connected components of $W^s \setminus \{p\}$, where $W^s$ is the
local stable manifold at $p$. Since $\xi$ and $\xi'$ are positively
proportional on $Q$, the separatrix $W^s\cap Q\subset S'$ is a trajectory
$|\gamma_q|$ of $\xi$ and thus $\omega(\gamma_q)=p$.
 Going backwards in time, $|\gamma_q|$ cuts
$\widetilde{C}_\varepsilon$ at a point $\rho(\vp_0)$ for some
 $\vp_0\in\SSS^1$ and thus $\widetilde{\Phi}(\vp_0,0)=p$.
Let $H:\clos(Q)\times[0,\delta[\to\wt{S}$, where $H(q,t)$ is the intersection
point of the trajectory of $\xi'$ through the point $q$ with the level curve
$\{u^lv^m=t\}$. As in the previous case, continuity at $p$ of the mapping
$\wt{\Phi}$ will follow from the continuity at
$p=(0,0)$ of the mapping $H$.  This property is easily obtained by explicit
computation when the vector field $\xi'$ is linear, and
we can reduce to this case using Hartman-Grobman Theorem
(see for instance  \cite{Pal-Mel}).
\end{proof}
\begin{definition}\label{def:uara}
Let $\Omega$ be a finite subset of $\CC$ (such as the
exceptional set of a parameterization $\wt{\Phi}$ in the
proposition above). An analytic mapping $F: \SSS^1\times]0,\eps] \to
N$, is called {\em uniformly almost ramified-analytic (with respect to $\Omega$)}
if there exists some $l\in\N$ such that $(\vp,r)\mapsto
F(\vp,r^l)$ can be extended as an analytic mapping at any point of
$\CC\setminus\Omega$. To be shorter, we will either write $\Omega$-u-a-r-a or simply u-a-r-a
if the subset $\Omega$ is understood.
\end{definition}
Part (iv) of Proposition~\ref{pro:parameterization-of-surfaces} says
that $\wt{\Phi}$ (or $\Phi$) is an u-a-r-a mapping with respect to
the exceptional set $\Omega$.
Since ramified-analyticity at any point of $\CC\setminus\Omega$
is inherited from the construction of $\Phi$ and uniformity comes
from the compactness of $E$, another typical situation example we will
come across in the sequel is the following: if $h$ is a continuous
function in a neighborhood of $E\subset M\times\R_{\geq 0}$ which is
ramified-analytic along $E$ (with respect to $D=M\times\{0\}$),
the composite mapping $h^\Phi=h\circ\Phi$ is $\Omega$-u-a-r-a.
\section{Asymptotic expansions of restricted functions}\label{section:expansion-function}

A {\em $\Q$-generalized (real) formal power-series} is a formal
expansion $ G(T) = \sum_{k \geqslant 0} a_k T^{\alpha_k}, $ where
$(\alpha_k)_{k\geqslant 0}$ is a strictly increasing sequence of
non-negative rational numbers and each coefficient $a_k$ is a real
number. It is said {\em convergent} if there exists $t_0>0$ such
that the sequence of $m$-partial sum functions $G_m:\R_{\geq
0}\to\R$, $G_m(t)=\sum_{k=0}^m a_k t^{\alpha_k}$, converges
uniformly in $[0,t_0]$, thus given rise to a continuous function,
also denoted $G:[0,t_0]\to\R$, analytic for $t> 0$, called
{\em the sum of the convergent series}. If the exponents $\alpha_k$
are in $\frac{\N}{l}$ for some positive integer $l$, then $G(T)$ is
called a {\em Puiseux series}. If all but finitely many coefficients
$a_k$ are non-zero then $G(T)$ is a {\em $\Q$-generalized real
polynomial}.

\smallskip
Let $X\subset\R^n$ be an analytic isolated surface singularity at
$\oo$ and let $S_0$ be a connected component of $X\setminus\{\oo\}$.
Let $\beta:M\times\R_{\geq 0}\to\R^n$ be an opening blowing-up of
$S_0$ and denote $S=\beta^{-1}(S_0)$, $D=\{z=0\}=M\times\{0\}$,
$E=\clos(S)\cap D$ as in the previous section.

\smallskip\noindent
Let $f:\UU\to\R$ be a continuous function in $\UU$, a neighborhood of
$E$ in $M\times\mathbb{R}_{\geq 0}$, which is {\it ramified-analytic
along $D$}. Let $f_S:\clos(S)\to\R$ be the restriction of $f$ to
$\clos(S)=S\cup E$. Given a $(S,z)$-resolution
$\RR=(\wt{S},\wt{E},\sigma)$ and an associated parameterization
$\wt{\Phi}:\SSS^1\times[0,\varepsilon]\to\wt{S}$ as in
Proposition~\ref{pro:parameterization-of-surfaces}, we denote by
$f^\Phi: = f_S\circ\Phi =f _S\circ\sigma\circ\wt{\Phi}:\SSS^1\times[0,\eps]\to\R$.

\smallskip
This Section is devoted to prove the following result,
establishing an asymptotic expansion of the restricted function
$f_S$ w.r.t. the height coordinate $z:M\times\R\to\R_{\geq 0}$ (let
again $(\vp,r)$ be the standard coordinates on
$\SSS^1\times[0,\eps]$).
\begin{proposition}\label{prop:(z,S)-relative-principal-expansion}
Assume that $f$ is not identically vanishing on $S$.
One and only one of the following two properties is satisfied:

\noindent (a) There exists a $\Q$-generalized real formal
power-series $G (T) = \sum_{k \geqslant 0} a_k T^{\alpha_k}$ which
is an {\em asymptotic expansion of $f_S$} in the following sense:
for any positive integer $m$, there exists a neighborhood $\VV_m$ of $E$ in
$\clos(S)$ and a bounded function $g_m:\VV_m\to\R$ such that, for any $(\yy,z)\in \VV_m$
with $z \neq 0$,
\begin{equation}\label{eq:asymptotic-expansion-of-fS}
 f_S(\yy,z) = \sum_{k=0}^{m-1} a_k z^{\alpha_k} +
 z^{\alpha_{m}}g_m(\yy,z).
\end{equation}
Moreover, the formal power series
$G(T)$ is a convergent Puiseux series and $f_S(\yy,z)=G(z)$ for any
$(\yy,z)\in S$ in a neighborhood of $E$.

\smallskip
\noindent
(b) Given an initial resolution of $S_0$, there exists a dominating
resolution $\RR^0$ adapted to the function $z$, a
$\Q$-generalized polynomial $P(T) = \sum_{k=0}^m a_k T^{\alpha_k}$
and a rational number $\alpha>\alpha_m$ such that, for any
resolution $\RR\succeq\RR^0$ and any associated parameterization
$\wt{\Phi}$, the mapping $f^\Phi:\SSS^1\times[0,\eps]\to\R$ writes
as
\begin{equation}\label{eq:expression-fR}
f^\Phi(\vp,r) = P(r) + r^{\alpha} F(\vp,r),
\end{equation}
where $F$ is a continuous function on $\SSS^1\times[0,\eps]$ and its
restriction to $\CC:=\SSS^1\times\{0\}$ is not constant. Moreover,
$F$ is u-a-r-a with respect to the exceptional set $\Omega$ of
$\Phi$.
\end{proposition}
The proof will follow from the following lemma.
\begin{lemma}\label{lm:principal-parts}
With the hypotheses and notations of Proposition
~\ref{prop:(z,S)-relative-principal-expansion}, we find:

(i) There exists a unique $\alpha=\alpha(f_S)\in\Q_{\geq 0}$ such
that the quotient $f_S/z^\alpha$ is bounded on $S$ and cannot have
the value $0$ as a single  accumulation value as $z\to 0^+$. The
number $\alpha$ is called the {\em exponent of the restricted
function $f_S$ (with respect to $E$)}.

(ii) Given an initial resolution of $S_0$, there exists a dominating
$(S,z)$-resolution $\RR^0$ such that, for any other resolution
$\RR\succeq\RR^0$ and any associated parameterization $\Phi:\SSS^1\times[0,\eps]\to S$,
the quotient function $f^\Phi/r^\alpha = (f/z^\alpha)\circ \Phi$ is well defined and analytic on
$\SSS^1\times]0,\eps]$ and extends to a continuous function on
$\SSS^1\times[0,\eps]$. Its restriction to the bottom circle
$\CC$ will be denoted by $\IN^\Phi(f)$ and called
the {\em initial part of the restricted function $f_S$ (relative to
$\Phi$).}

(iii) An initial part $\IN^\Phi(f)$ like in point (ii) is constant
if and only $f_S/z^\alpha$ has a unique accumulation value as $z \to 0$.
\end{lemma}
\begin{proof}
\noindent By definition  of a ramified-analytic function along $D$
and since $E$ is a compact subset of $D$, there exists a positive
integer $l\in\N$ such that the function $\overline{f}:({\bf
y},z)\mapsto f({\bf y},z^l)$ is analytic in a neighborhood of $E$ in
$M\times\R$. If we prove the Lemma for the analytic function
$(\overline{f})_S := \overline{f}|_S$, we obtain the exponent $\overline{\alpha}$.
Then $\alpha:=\bar{\alpha}/l$ is the exponent of $f_S$ with respect to $E$ and it
satisfies (i)-(iii). For the rest of the proof, we suppose that
$f$ is analytic in a neighborhood of $E$ in $M\times\mathbb{R}$.

\smallskip
\noindent Proof of (i). The uniqueness of the exponent $\alpha$ is
immediate from its definition. Consider the following function
\begin{center}
\vspace{4pt}
$\mu (t) = \max\{|f(\yy,t)|$ for $(\yy,t) \in S\}$.
\vspace{4pt}
\end{center}
It is well defined since $S \cap \{z=t\}$ is compact for $t>0$. The function
 $\mu$ is subanalytic, continuous and identically
zero only if $f_S$ is. So assuming that $f_S$ does not vanish
identically on $S$, there exists a positive real number $a$ and a
non-negative rational number $\alpha$ such that $t^{-\alpha} \mu (t)
\to a$ as $t\to 0$. This proves the claim.

\smallskip
\noindent Proof of (ii). Assume we are given a first resolution.
Let $\RR^0$ be a $(S,f,z)$-resolution dominating it.
Any other resolution $\RR=(\wt{S},\wt{E},\sigma)$ dominating
$\RR^0$ is still a resolution adapted to $f$ and $z$. Let
$\wt{\Phi}:\SSS^1\times[0,\eps]\to\wt{S}$ be a parameterization
associated with $\RR$. The function $h=z^{-\alpha}f_S$ is analytic,
continuous and bounded on $S\cap\{0<z<\eps\}$ for some $\eps>0$. Let
$S'=\sigma^{-1}(S)$, $E'=\clos(S')\setminus S'$ be respectively the
strict transform of $S$  and the strict divisor of the resolution
$\RR$ (see the notations of Theorem~\ref{th:reduction-singularities}).
Let $h^\prime=h\circ\sigma:S'\to\R.$ Thus
$r^{-\alpha}f^\Phi=h^\prime\circ\wt{\Phi}$. Since $\wt{\Phi}$ is
continuous and maps $\CC$ onto $E'$, there is just to prove that
$h^\prime$ extends to a continuous function up to $E^\prime$. We
also write $h^\prime=\wt{z}^{-\alpha}\wt{f_S}$ where
$\wt{f_S}=f_S\circ\sigma$ and $\wt{z}=z\circ\sigma$.

\smallskip
\noindent
First, let $p'\in E'\setminus \sing (\wt{E})$.
There are analytic coordinates $(u,v)$ of $\wt{S}$ at $p'$ such that $\wt{E}=\{v=0\}$ and
$\{v>0\}\subset S'$ using (iv) of Theorem \ref{th:reduction-singularities}. Since $\RR$
is a resolution adapted to $f$ and $z$, we write
\begin{center}
\vspace{4pt}
$\wt{f_S}(u,v)=u^{l_1}v^{m_1}U_1(u,v),\;\;\wt{z}=v^{m_2}U_2(u,v)$
\vspace{4pt}
\end{center}
for some integers $l_1,m_1,m_2\in\N$ and invertible analytic functions $U_1,U_2$ with
$U_2(0,0)>0$. For $(u,v)$ close to $p'=(0,0)$ with $v>0$, we find
\begin{equation}\label{eq:tilde-h-in-regular-point}
h^\prime(u,v)=v^{m_1-\alpha
m_2}\frac{u^{l_1}U_1(u,v)}{U_2(u,v)^\alpha}.
\end{equation}
Since $h^\prime$ is bounded on $\{v>0\}$ necessarily  $m_1\geq\alpha m_2$
and the right hand term in Equation (\ref{eq:tilde-h-in-regular-point})
defines a continuous function on $\{v\geq 0\}$. If $\clos(S')\subset\{v\geq 0\}$
nearby $p'$, we get the desired conclusion. If instead $\{v<0\}\subset S'$,
necessarily $m_2$ is even since $\wt{z}$ is positive on $S'$. In this case,
the monomial $v^{m_1-\alpha m_2}$ in expression
(\ref{eq:tilde-h-in-regular-point}) must be read
as $v^{m_1}/(v^{m_2})^\alpha$. The function $h'$ turns out to be
continuous in a neighborhood of $p' = (0,0)$.

\smallskip\noindent
Suppose now that $p'\in E'\cap\sing\wt{E}$. There
are analytic coordinates $(u,v)$ of $\wt{S}$ at $p'$ with
$\wt{E}=\{uv=0\}$ and  $\{u>0,v>0\}\subset S'$ and such that we can
write
\begin{center}
\vspace{4pt}
$\wt{f_S}(u,v)=u^{l_1}v^{m_1}U_1(u,v),\;\;\wt{z}=u^{l_2}v^{m_2}U_2(u,v)$
\vspace{4pt}
\end{center}
for some $l_1,m_1,l_2,m_2\in\N$ and analytic functions $U_1,U_2$ with
$U_1(0,0)\neq 0$, $U_2(0,0)>0$. This time, for small and positive $u,v$, we have
\begin{equation}\label{eq:tilde-h-in-singular-point}
h^\prime(u,v)=u^{l_1-\alpha l_2}v^{m_1-\alpha
m_2}\frac{U_1(u,v)}{U_2(u,v)^\alpha}.
\end{equation}
Since the function $h'$ is bounded in a neighborhood of $p'$ in $S'$,
$l_1-\alpha l_2$ and $m_1-\alpha m_2$ are both non-negative.
The continuity of $h'$ follows by the same arguments as in the previous case.

\smallskip\noindent
Proof of (iii). It follows by continuity of $f^\Phi/r^\alpha$, proved in (ii),
the properness of $\Phi$ and that $\Phi$ maps $\mathcal{C}$ onto $E=\clos(S)\cap\{z=0\}$.
\end{proof}

\noindent {\em Proof of
Proposition~\ref{prop:(z,S)-relative-principal-expansion}.} Let
$\alpha_0\in\Q_{\geq 0}$ be the exponent of $f$ with respect to $E$.
Let $\RR^0$ be a $(S,z)$-resolution
and $\Phi^0$ be an associated parameterization satisfying
the properties of  (ii) in Lemma~\ref{lm:principal-parts}.
\\
If the initial part $\IN^{\Phi^0}(f)$ is not constant then
we are in case (b) of the proposition with $P=0$ and $\alpha=\alpha_0$.
\\
Assume now
$\IN^{\Phi^0}(f) \equiv a_0\in\R^*$. The function $f_1: = f-a_0z^{\alpha_0}$ is
ramified-analytic along $D$.  If
$f_1|_S\equiv 0$ then we are in case (a). Otherwise,
using Lemma~\ref{lm:principal-parts}, let $\alpha_1\in\Q_{\geq 0}$
be the exponent of $f_1$ w.r.t $E$. By definition of the exponent,
we find $\alpha_1>\alpha_0$. Let $\RR^1$ be a $(S,z)$-resolution
with $\RR^1\succeq\RR^0$ and $\Phi^1$ an associated
parameterization for which the initial part $\IN^{\Phi^1}(f_1)$ of
$f_1$ exists as in part (ii). If $\IN^{\Phi^1}(f_1)$ is not constant
we are in case (b) as above and we are done, otherwise we continue
this process.

\noindent Suppose there exists a sequence of $(S,z)$-resolutions
$\{\RR^k\}_{k\geq 0}$ with $\RR^{k+1}\succeq\RR^k$, associated
parameterizations $\Phi^k$ and a $\Q$-generalized  power series
$G(T)=\sum_{k\geq 0}a_kT^{\alpha_k}$ such that, for any $m\geq 0$,
$\alpha_m$ is the exponent of the function
$f_m=f-\sum_{k=0}^{m-1}a_kz^{\alpha_k}$ and the principal part $in^{\Phi^m}(f_m)$
is a constant function equal to
$a_m \neq 0$. The definition of the exponent $\alpha$ gives
directly the asymptotic expansion of $f_S$ as in equation
(\ref{eq:asymptotic-expansion-of-fS}).
Let $\Gamma\subset S$ be an analytic half-branch accumulating to a
single point in $E$, parameterized by the variable $z$. Let
$L:]0,\eps]\to\R$ defined as $L(z)=f_S(\Gamma(z))$.
By (\ref{eq:asymptotic-expansion-of-fS}), we have for any $m\geq 0$ and
$z$ sufficiently small,
\begin{center}
\vspace{4pt}
$L(z)-\sum_{k}^{m-1}a_kz^{\alpha_k}=O(z^{\alpha_m}),$
\vspace{4pt}
\end{center}
that is, that $G(T)$ is the asymptotic expansion of $L$ as $z\to
0^+$. Since $L$ is a semi-analytic function, $G(T)$
is a convergent Puiseux series. Thus $L(z)=G(z)$, where $G$ is
considered here as the sum of the expansion $G(T)$. We define
$G_S:S\to\R$ by $G_S(\yy,z)=G(z)$, an analytic function on $S$ which
depends only on $z$. We have shown that the restrictions of $f_S$
and $G_S$ on $\Gamma$ coincide. Since $\Gamma$ can be chosen
arbitrarily, $f_S=G_S$ on the whole surface $S$. This proves
statement (a) of the Proposition.\\
Finally, $F=\frac{f^\Phi-P}{r^\alpha}$ is u-a-r-a since both
$f^\Phi$ and $P$ are so. \hfill{$\square$}
\begin{remark}\label{rm:about-F}{\em
Although $F$ depends on the resolution $\RR$ and on the
associated parameterization $\Phi$, we insist it is of the special following form:
\begin{center}
$F=g\circ\Phi\;\;$ with $\;\;g := (\frac{f^\Phi-P}{r^\alpha})$.
\end{center}
The function $g$ is continuous in a neighborhood of $E$ in $M\times\R_{\geq
0}$, ramified-analytic along $E$, and depends on $f$ and the
opening blowing-up $\beta$ only.}
\end{remark}
\section{Oscillation vs Spiraling in singular surfaces}\label{section:dynamic-on-cylinder}
Let $\gamma:[0,+\infty[\to\R^n$ be an analytically parameterized curve
such that $\lim_{t\to +\infty}\gamma(t)=\oo\in\R^n$ and $\oo$ does not belong to
$\absgamma$, the image of  $\gamma$.
\begin{definition}
A parameterized curve $\gamma$ is said {\em (analytically)
non-oscillating} if for any semi-analytic subset $H$ of $\R^n$,
either $\absgamma$ is contained in the subset $H$ or the
intersection $\absgamma\cap H$ consists at most of finitely many
points. If, on the contrary, there exists a semi-analytic set $H$
such that $\absgamma$ is not contained in $H$ and the intersection
$\absgamma\cap H$ has infinitely many points then we will say that
$\gamma$ is {\em oscillating relatively to $H$}.
\end{definition}
The notion of oscillation clearly depends only on the germ at $\oo$ of
the image $\absgamma$ of the parameterized curve $\gamma$, not on any
given parameterization.

\medskip
In dimension $2$, the notion of {\it spiraling} around a given point
is a special case of oscillation for a curve.
A convenient definition is found in \cite{Can-M-S}. We generalize this
notion for a curve $\absgamma$ contained in an analytic isolated surface singularity
$X\subset \R^n$ at the origin $\oo$ and accumulating at $\oo$.

\smallskip\noindent
Let $X$ be an analytic surface with an isolated singularity
at $\oo\in\R^n$. Let $S_0$ be a connected component of
$X\setminus\{\oo\}$. Let $\Gamma$ be an analytic half-branch at $\oo$ contained in
$S_0$. For a small enough simply connected neighborhood $\VV$ of (the germ at $\oo$ of)
$\Gamma$ in $S_0$, the curve $\Gamma\cap \VV$ separates $\VV\setminus \Gamma$ into
two connected components which we call the \textit{two local sides of $\Gamma$ in $S_0$}.
\begin{definition}
The curve $\gamma:[0,+\infty[\to S_0\subset X\setminus\{\oo\}$ {\em
spirals in $X$} if, for any analytic half-branch $\Gamma$ at $\oo$
in $S_0$, there exists an increasing sequence $(t_k)_{k\in\N} \subset \R_{> 0}$
with $t_k \to +\infty$ such that for each $k$:
\begin{center}
\vspace{4pt}
$\gamma ([t_k,t_{k+1}[) \cap \Gamma = \{\gamma (t_k)\}$,
$\;\;\gamma(t_k-\eps_k) \in \VV^-\;$ and
$\;\gamma(t_k+\eps_k) \in \VV^+$,
\vspace{4pt}
\end{center}
for $\eps_k >0$ small and where $\VV^-$, $\VV^+$ are the local sides of $\Gamma$ in $S_0$.
\end{definition}
When $\gamma$ is a trajectory of a real analytic vector field in a
neighborhood of $\oo\in\R^2$, a Rolle-Khovanskii's argument proves
that the only oscillating dynamics at $\oo$ is spiraling (see
\cite{Can-M-S}). Proposition~\ref{pro:oscillation=spiraling} below
extends this result to analytic isolated surfaces
singularities.

\smallskip\noindent
Let $\xi_0$ be an analytic vector field on $S_0$ which extends
continuously and subanalytically to the origin by $\xi_0(\oo)=0$,
as a mapping from $\clos(S_0)$ to $T\R^n|_{\clos(S_0)}$.
\begin{proposition}\label{pro:oscillation=spiraling}
 Assume that
$\xi_0$ does not vanish in $S_0$. Let $\gamma:[0,+\infty[\to S_0$ be
a non-trivial trajectory of $\xi_0$ accumulating at $\oo$. Then
$\gamma$ is oscillating if and only if it spirals in $X$.
\end{proposition}
\begin{proof}
If $\gamma$ spirals then it is oscillating. Suppose that
$\gamma$ does not spiral. There exists an analytic half-branch
$\Gamma$ in $S_0$ such that either

\noindent (a) the germ at $\oo$ of the intersection
$\absgamma\cap\Gamma$ is empty, or

\noindent (b) $\absgamma\cap\Gamma$ is infinite but $\gamma$ does
not cross $\Gamma$ from one fixed local side of $\Gamma$ to the other side
at those intersection points.

\noindent If (b) happens, a Rolle's argument implies that $\Gamma$
is tangent to $\xi_0$ at infinitely many points accumulating
to $\oo$. The subanalyticity of $\xi_0$ implies that the
half-branch $\Gamma$ is a trajectory of $\xi_0$, contradicting the
oscillation of $\gamma$ relatively to $\Gamma$. So (b) is
impossible.

\noindent Assume we are in case (a). Since the surface $S_0$ is
analytically diffeomorphic to a cylinder, $S_0\setminus\Gamma$ is a
simply connected analytic manifold. Using Haefliger's Theorem
\cite{Hae,Kho,Mou-R}), we deduce that any leaf of the real analytic
foliation induced by $\xi_0$ in $S_0\setminus\Gamma$ is a {\em
Rolle's leaf}.
In particular, the curve $\absgamma
\subset S_0\setminus\Gamma$ is a Rolle's leaf and cannot cut
infinitely many times any analytic half-branch contained in
$S_0\setminus\Gamma$. Thus $\gamma$ is non-oscillating.
\end{proof}

Despite of the similarities between spiraling in a smooth surface and in an analytic
isolated surface singularity, there is however a very important difference.
The existence, for a trajectory $\gamma$, of a {\em tangent at the origin}, that is
the limit of secants $\lim_{t\to\infty}\frac{\gamma(t)}{|\gamma(t)|}$ exists,
prevents, in the smooth surface situation, from spiraling around the origin.
For an isolated surface singularity, although in the OTC case this argument is still valid,
in the CTC situation, the curve $\gamma$ will always have a tangent
at the origin corresponding to the direction of the tangent cone, regardless
if it is spiraling or not

\noindent A criterion stronger than the existence of tangent
to imply non-spiraling is that the lifting of $\gamma$ by a
reduction of singularities of the surface accumulates to a single
point on the exceptional divisor.
\\
We will use this criterion through its lifting on $\SSS^1\times [0,\eps]$
via a parameterization as in
section~\ref{section:ParameterizationSurfaces}.

\smallskip
\noindent {\bf Criterion for non-spiraling.}
Let $\RR$ be a resolution of $S=\beta^{-1}(S_0)$ where $\beta$ is an opening
blowing-up of $S_0$. Let $\Phi:\SSS^1\times[0,\eps]\to S$ be a parameterization
associated to $\RR$.
Assume that $\absgamma\subset S_0$ and suppose the $\omega$-limit set
$\omega(\overline{\gamma})$ of the lifted curve $\overline{\gamma}=
(\beta\circ\Phi)^{-1}\circ\gamma$ is such that $\CC\setminus\omega(\overline{\gamma})$
contains an open non-empty arc. Then $\gamma$ does not spiral in $X$.

\smallskip\noindent
The proof is easy: the stated property will imply that
$\overline{\gamma}$ does not intersect a given analytic half-branch
$\overline{\Gamma}$ on $\SSS^1\times[0,\eps]$ through a point
$p\in\CC\setminus(\omega(\overline{\gamma})\cup\Omega)$ where
$\Omega$ is the exceptional set of $\Phi$. Therefore, $\gamma$ does
not intersect the curve
$\Gamma=(\beta\circ\Phi)(\overline{\Gamma})\subset S_0$, which is an
analytic half-branch by properness of the resolution and the
property that $\Phi$ is ramified-analytic at $p$. Thus $\gamma$ does
not spiral in $X$.

\medskip
The next result describes, for a vector field $\xi_0$ on $S_0$, two types of dynamics
ensuring that none of its trajectories accumulating at the origin is spiraling.
These types correspond to
either ``dicritical'' or ``non-monodromic'' dynamics similar to those in the
plane gradient case met in Section \ref{Section:plan-article}.
\begin{proposition}\label{pro:dic-or-nonmon}
Assume that $\xi_0$ does not vanish in $S_0$. Suppose that the
transformed vector field $\overline{\xi}=(\beta\circ\Phi)^\ast\xi_0$
on the open cylinder $\SSS^1\times]0,\eps]$ satisfies one of the
following non-exclusive situations:

\smallskip\noindent
\em (a) Dicritical case: \em There exist a point $p\in \CC \setminus
\Omega$ and a neighborhood $\UU$ of $p$ in $\SSS^1\times[0,\eps]$
disjoint from $\Omega$ in which $\overline{\xi}$ writes as
\vspace{4pt}
\begin{equation}\label{eq:dic}
\vspace{6pt} \left \{
\begin{array}{rcl}
\dot{r} & = & r^{\mu}H(r,\vp) \\
\dot{\vp} & = & r^{\mu-1+\eta}G(r,\vp)
\end{array}
\right. \vspace{4pt}
\end{equation}
where $\mu,\eta\in\Q_{>0}$ and $H,G$ are continuous on $\UU$ and
ramified-analytic at any point of $\UU\cap\CC$ and such that $H$ is
negative on $\UU$.

\smallskip\noindent
\em (b) Non-monodromic case: \em There exist $\mu\in\Q_{\geq 0}$,
u-a-r-a functions $G_1,G_2:\SSS^1\times]0,\eps]\to\RR$ so that $G_2$
vanishes on $\CC\setminus\Omega$ and an u-a-r-a
function $H$ continuous on the whole cylinder
$\SSS^1\times[0,\eps]$ such that the restricted function
$H|_{\CC}$ is not constant, in such a way that $\overline{\xi}$
writes in the open cylinder $\SSS^1\times]0,\eps]$ as
\vspace{4pt}
\begin{equation}\label{eq:nonmon}
\vspace{6pt} \left \{
\begin{array}{rcl}
\dot{r} & = & r^{\mu+1}G_1\\
\dot{\vp} & = &r^{\mu}[\frac{\partial H}{\partial \vp} + G_2].
\end{array}
\right.
\end{equation}

\vspace{4pt}\noindent
 Then any trajectory $\gamma$ of $\xi_0$ accumulating
to the origin is non-spiraling and therefore is non-oscillating.
\end{proposition}

\begin{proof}
It suffices to show that any trajectory $\gamma$ of $\xi_0$
accumulating to the origin satisfies the non-spiraling criterion above.

\smallskip\noindent In  the dicritical situation (a) we prove a
slightly stronger result: there exists a non-empty arc
$I\subset\UU\cap\CC$ such that each point in $I$ is the unique
$\omega$-limit point of a trajectory of the transformed
vector field $\overline{\xi}$. \\
When $\mu-\eta+1\geq\mu$ in Equation (\ref{eq:dic}), dividing $\overline{\xi}$
by $r^\mu$, gives a vector field which extends to $\UU\cap\SSS^1\times[0,\eps]$
as a ramified-analytic vector field transverse to $\CC\cap\UU$. Thus any point
of $\CC\cap\UU$ is the unique accumulation point of a trajectory of $\overline{\xi}$
living in the open cylinder $\SSS^1\times]0,\eps]$.
\\
Assume now that $\mu-\eta+1<\mu$ in (\ref{eq:dic}). We suppose that $\UU$
is of the form $\UU=]\vp_1,\vp_2[\times[0,\delta]\in\SSS^1\times[0,\eps]$ for
some $\delta>0$ small enough. Dividing $\overline{\xi}$ by $r^{\mu-1}|H|$,
our vector field provides the following equations in
$\UU$:
\vspace{4pt}
\begin{equation}\label{eq:dic-bis}
 \left \{
\begin{array}{rcl}
\vspace{4pt}
\dot{r} & = & -r \\
\dot{\vp} & = & r^{\eta}\frac{G}{\mid H \mid}
\end{array}
\right.
\vspace{4pt}
\end{equation}
Up to shrinking $\UU$, and since $G$ is ramified-analytic,
we furthermore assume that $G$ does not
vanish on $\UU$, up to increasing the exponent $\eta$.
If $G(p)= 0$ but $G|_{\UU\cap\CC}\not\equiv 0$,
then there are points of $\CC$ close to $p$ at which $G$ does not vanish.
Thus we can also suppose $G(p)\neq 0$, for instance that $G$ is positive on $\UU$.
Up to shrinking $\UU$ again, we know that $K_1\leq \frac{G}{\mid H\mid}\leq K_2$
on $[\vp_1,\vp_2]\times[0,\delta]$ for some positive constants
$K_1,K_2$. The solution of (\ref{eq:dic-bis}) through a point
$(\vp_0,r_0)\in [\vp_1,\vp_2]\times ]0,\delta]$, as long as it is in
that domain, lies between the solutions through $(\vp_0,r_0)$ of the
systems of equations $\dot{r} = -r$ and $\dot{\vp} = K_ir^\mu$ for
$i=1,2$. These last curves are parameterized by
\begin{center}
\vspace{4pt} $ \vp\mapsto
r(\vp)=[r_0^\eta-\frac{\eta}{K_i}(\vp-\vp_0)]^{1/\eta}, \;i=1,2.$
\vspace{4pt}
\end{center}
We deduce that any point of $]\vp_1,\vp_2[\times 0\subset\CC$ is the unique accumulation
point of a trajectory of the system (\ref{eq:dic-bis}), lying in $\{r>0\}$.

\medskip \noindent Consider now the non-monodromic situation (b).
The hypothesis about $H$ implies its partial derivative $\partial_\vp H$ is u-a-r-a and continuous
along $\CC\setminus\Omega$. Let $\crit^*(H|_\CC)$ be the critical
locus of $H|_\CC$ in $\CC\setminus\Omega$, and let
\begin{center}
\vspace{4pt}
$\Omega'=\Omega\cup (H|_\CC)^{-1}(H(\crit^*(H|_\CC))$.
\vspace{4pt}
\end{center}
Since $H|_\CC$ is not constant, $\CC\setminus\Omega'$ has non empty
interior. To show the criterion for non-spiraling for any trajectory
$\gamma$ of $\xi_0$, it is enough to check that the limit set
$\omega(\overline{\gamma})$ of any trajectory $\overline{\gamma}$ of
$\overline{\xi}$ accumulating to $\CC$ is
contained in $\Omega'$.
\\
Assume $\overline{\gamma}$ is parameterized by $t\in\R_{\geq 0}$ and
consider the real function
\begin{center}
\vspace{4pt}
$h=h_{\overline{\gamma}}:\R_{\geq 0}\to\R,\;\;t\mapsto
h(t)=H(\overline{\gamma}(t)).$
\vspace{4pt}
\end{center}
The function $h$ is $C^1$. Let $p\in\CC\setminus\Omega'$ and let
$a:=H(p)$. We just have to show that $a$ cannot be an accumulation
value of $h$ when $t\to +\infty$. The function $|\partial_\vp H|$ is
bounded below on the compact set $(H|_\CC)^{-1}(a)$: there exists
$c>0$ such that $|\partial_\vp H |\geq 2c>0$ on a given
neighborhood $\VV$ of $(H|_\CC)^{-1}(a)$ in $\SSS^1\times[0,\eps]$.
Since $\overline{\gamma}(t)=(\vp(t),r(t))$ satisfies Equations
(\ref{eq:nonmon}),  if $\overline{\gamma}(t)\in\VV$ then, up to shrinking $\VV$
(taking into account Remark~\ref{rk:derivative-wrt-r}), we find
\begin{center}
\vspace{4pt} $\dot{h}(t)= \displaystyle{\frac{\partial H}{\partial
r}(\overline{\gamma}(t))\,\dot{r}(t)+\frac{\partial H}{\partial
\vp}(\overline{\gamma}(t))\,\dot{\vp}(t)} >c.$ \vspace{4pt}
\end{center}
For $\eps'>0$ sufficiently small, we assume that
$H^{-1}(a)\cap\SSS^1\times[0,\eps']\subset \VV$. Thus, there exists
some $\delta>0$ such that
\begin{center}
\vspace{4pt}
$h(t)\in]a-\delta,a+\delta[\Rightarrow \overline{\gamma}(t)\in \VV$.
\vspace{4pt}
\end{center}
From all these properties, for $t\in h^{-1}(]a-\delta,a+\delta[)$ large enough,
$\dot{h}(t)\geq c/2>0$. Thus when $t\to+\infty$, the value $a$ cannot be an
accumulation value of $h$.
\end{proof}
\begin{remark}\label{rk:criterion}{\em
The non-monodromic situation (b) described in
Proposition~\ref{pro:dic-or-nonmon} can be generalized as follows:\\
Assume that the given analytic vector field $\xi_0$ does not vanish
in $S_0$ and that the foliation $\overline{\mathcal{F}}$ on $\SSS^1\times]0,\eps]$
induced by $\overline{\xi}=(\beta\circ\Phi)^\ast\xi_0$ extends continuously
to $\CC\setminus\Omega$ such that $\CC$ is invariant. Assume there exist,
two distinct points $q_1,q_2$ of $\CC\setminus\Omega$ where
$\overline{\mathcal{F}}$ is not singular, and two continuous germs of
vector fields $\xi_1$ at $q_1$ and $\xi_2$ at $q_2$,
which are local generators of the foliation
$\overline{\mathcal{F}}$, positively co-linear to $\overline{\xi}$
in the common domain of definition and "pointing in different
directions": if $\vp$ denotes a global coordinate on
$\CC\simeq\SSS^1$, writing $\xi_i (q_i) = c_i\partial_{\vp}$, then
$c_1 c_2<0$. Then any trajectory of $\xi_0$ accumulating to the
origin is non-oscillating.}
\end{remark}
\section{Proof of the main result}\label{section:ProofMainResult}
This section is devoted to the proof of the main result of this
paper, Theorem~\ref{thm:MainResult}.

The next Lemma shows that the only case requiring work is when both
$f_0|_{S_0}$ and $\nbh f_0$ do not vanish on $S_0$.

\begin{lemma}\label{lm:vanishing-case}
If either $f_0|_{S_0}$ or $\nbh f_0$ vanishes in
any neighborhood of $\oo$ in $S_0$, then any trajectory of the
restricted gradient $\nbh f_0$ accumulating to the origin is
non-oscillating.
\end{lemma}
\begin{proof}
First, note that $\nbh f_0$ extends to a continuous subanalytic mapping $\nbh
f_0:\clos(S_0) \to T\R^n$ by $\nbh f_0(\oo)=\oo$, and that $f_0$ vanishes
on any connected component of the zero locus of $\nbh f_0$ containing $\oo$ in
its closure. The subanalytic Curve Selection Lemma
guarantees there exists a subanalytic (thus semi-analytic) curve $\Gamma\subset S_0$
such that $\oo\in\clos(\Gamma)$ and $\Gamma\subset f^{-1}_0(0)$. Let $\gamma$
be a non-trivial trajectory of the restricted gradient $\nbh f_0$
accumulating to $\oo$. The function $t\mapsto f_0(\gamma(t))$ is
increasing and tends to $0$ as $t\to\infty$. Thus $f_0(\gamma(t))<0$
for any $t$ and $\gamma$ does not cut $\Gamma$. Apply now
Proposition~\ref{pro:oscillation=spiraling}.
\end{proof}

Assume from now on that there exists a neighborhood $\VV$ of $\oo$
in $X$, such that $f_0|_{S_0}$ and $\nbh f_0$ do not vanish
in $\VV\cap S_0$.

\smallskip
The sketch of the proof of Theorem~\ref{thm:MainResult} is as
follows. We first open the surface $S_0$ by means of a suitable
opening blow-up mapping $\beta:M\times\R_{\geq 0}\to\R^n$ as defined in
Proposition~\ref{pro:opening-blow-up}. Then we take a suitable resolution
$\RR = (\wt{S},\wt{E},\sigma)$ of the surface $S=\beta^{-1}(S_0)$ as
in Theorem~\ref{th:reduction-singularities}. We then pick a
parametrization $\wt{\Phi}: \SSS^1\times[0,\eps] \to \wt{S}$
associated to $\RR$. Writing $\Phi=\sigma\circ\wt{\Phi}$, the
mapping $\beta\circ\Phi$ is a diffeomorphism from the open cylinder
$\SSS^1\times]0,\eps]$ onto $S_0$. Thus, the pull-back $\wt{\bh} := (\beta \circ\Phi)^*\bh$
of the metric $\bh$ is an analytic Riemannian metric on the open cylinder.
If $f^\Phi$ denotes the composition $f_0\circ\beta\circ\Phi$, then
the pull-back $\bar{\xi}:= (\beta\circ\Phi)^*\nabla_{\bh}f_0$
 is just the gradient vector field of $f^\Phi$ with respect to $\wt{\bh}$, that is,
\begin{center}
\vspace{3pt}
$
\bar{\xi}=(\beta\circ\Phi)^\ast\nabla_\bh
f_0=\nabla_{\wt{\bh}}f^\Phi.
$
\vspace{3pt}
\end{center}
The proof will be finished, using
Proposition~\ref{pro:dic-or-nonmon}, once we have proved that
$\overline{\xi}$ satisfies one of the two situations described there:
either (a), dicritical or (b), non-monodromic.

\smallskip
Our proof will only deal with the metric $\bg$ on $\R^n$ be the Euclidean
metric. We can reduce to this case using Cartan-Janet's Theorem \cite{Jan,Car}:
an analytic Riemannian manifold can be locally isometrically embedded into an
Euclidean space as an analytic submanifold equipped with the induced
Riemannian structure.

\medskip\noindent
{\bf Notation.} Let $\Omega$ be a finite subset of $\CC$ (such as
for instance the exceptional set of a parameterization as in
Proposition~\ref{pro:parameterization-of-surfaces}). In
Definition~\ref{def:uara} was introduced the notion $\Omega$-u-a-r-a
function on the open cylinder $\SSS^1\times]0,\eps]$.
For any rational number $\nu\geq 0$, let $\AAA_{\geq\nu}$ be the
real algebra of all the $\Omega$-u-a-r-a functions
$\psi:\SSS^1\times]0,\eps]\to \R$ for which the function
$r^{-\nu}\psi$ is also an $\Omega$-u-a-r-a function along $\CC$. Let
$\AAA_{>\nu}:= \cap_{\mu>\nu}\AAA_{\geq\mu}$ be the ideal of
$\AAA_{\geq\nu}$ of the functions $\psi$ such that the function
$r^{-\nu}\psi$ vanishes identically on $\CC\setminus\Omega$. In
particular $\psi \in \AAA_{>\nu}$ means there exists a rational
number $\nu^\prime > \nu$ such that  $\psi \in
\AAA_{\geq\nu^\prime}$.

\medskip
We are dealing first with the CTC case in rather detailed
fashion. It requires much more work than the OTC case, and this
latter will follow from exactly the same arguments as those used in
the CTC case.

\medskip
\noindent{\bf Cuspidal case.} \\
Assume that the tangent cone of $S_0$ at the origin is
reduced to a single point. \\
Take linear coordinates $(x_1,\ldots,x_{n-1},z)$ at $\oo$ adapted to
$S_0$ which are also orthonormal coordinates for the Euclidean
metric.

\smallskip\noindent
We consider an opening blow-up mapping of the form
$$
\beta:M\times\R_{\geq 0} \to \R^n,\;\;(\yy,z) \to
(z^{eN}\yy+\theta(z),z^N),
$$
where $M = \R^{n-1}$ and $e,N,\theta$ are defined as in
Proposition~\ref{pro:opening-blow-up}. We recall that
$z\mapsto(\theta(z),z)=(\theta_1(z),\ldots,\theta_{n-1}(z),z^N)$ is
a parametrization of an analytic half-branch in $S_0$. Let $m+1\in\N_{\geq 1}$
be the minimum order at $0$ with respect to $z$ of the components $\theta_j$.
The cuspidal nature of the surface $S_0$ implies
\begin{center}
\vspace{4pt}
$m\geq N$.
\vspace{4pt}
\end{center}

We define the following functions on the open cylinder
\begin{center}
\vspace{4pt} $R=eN(y_1^2+\cdots+y_{n-1}^2)\circ\Phi\;$ and
$\;U=\sum_j ((y_j\circ\Phi)_\vp)^2 $, \vspace{4pt}
\end{center}
where the subscript $\vp$ stands for partial derivative with respect
to the angular variable $\vp$. Again $R$ and $U$ depend on the
resolution and on the associated parameterization $\Phi$ considered,
but both are u-a-r-a functions with respect to the resulting
exceptional set $\Omega$.
\begin{lemma}\label{lm:U}
There is a non-empty open arc $J$ of $\CC\setminus\Omega$ with non
empty interior in $\CC\setminus\Omega$ along which the restricted
function $U|_J$ is positive.
\end{lemma}
\begin{proof}
We just have to show that the function $U$ does not vanish on
the whole of $\CC\setminus\Omega$.
If $U|_{\CC\setminus\Omega}\equiv 0$, by definition of $U$, each $y_j\circ\Phi$
is locally constant when restricted to $\CC\setminus\Omega$, and thus constant on
$\CC$ by continuity. Using (iii) of
Proposition~\ref{pro:parameterization-of-surfaces}, this would imply
the constancy of the coordinates $y_j$ along $E=\sigma(\wt{\Phi}(\CC))$, which is
impossible since by construction $\dim E = 1$.
\end{proof}

Using the coordinates $(\vp,r)$ in the open cylinder, the metric $\wt{\bh}$ writes
\vspace{4pt}
\begin{equation}\label{eq:metric-h}
\wt{\bh}=(\beta\circ\Phi)^\ast{\bf
g}=A(r,\vp)dr^2+2B(r,\vp)drd\vp+C(r,\vp)d\vp^2.
\end{equation}

\smallskip
The following lemma describes the coefficients of
the metric $\wt{\bh}$. The dominant part of each term of interest
is explicit. It is important to remark that the
statement neither deals with a fixed resolution $\RR$, nor an associated
parametrization. It is very useful and necessary to obtain the needed
conclusions up to dominating resolutions of a given one as we will see.
\begin{lemma}\label{lm:expression_transf_metric}
There exists a resolution $\RR^0$ of $S$ such that, for any other
resolution $\RR\succeq\RR^0$ and parameterization $\Phi$ associated
to $\RR$, we have the following description: There exist
$s\in\Q_{\geq 0}\cup\{+\infty\}$ with \vspace{4pt}
\begin{equation}\label{eq:estimates_s}
s\geq eN +m, \vspace{4pt}
\end{equation}
an analytic power series $\psi(r)$ with $\psi(0)=0$ and an u-a-r-a
function $H$ on $\SSS^1\times]0,\eps]$ which extends continuously to
the circle $\CC$ in such a way that $H|_C$ is not constant, such
that we obtain the following expressions for the coefficients of
$\wt{\bh}$ in (\ref{eq:metric-h}):
\begin{equation}\label{eq:coefficients_metric}
\left\{
\begin{array}{rcl}
\vspace{4pt} A & = &r^{2N-2}[N^2+\psi(r)]
+r^{eN+m-1}\overline{A_1}+ r^{2eN-2}\overline{A_2}\\
\vspace{4pt}
B & = & r^sH_\vp + r^{2eN-1}R_\vp + \overline{B} \\
\vspace{4pt} C & = & r^{2eN}U
\end{array}
\right.
\end{equation}
where $\overline{A}_1,\overline{A}_2\in\AAA_{\geq 0}$ and
$\overline{B}\in\AAA_{>2eN-1}$ and with the convention that the term
$r^sH_\vp \equiv 0$ if $s=\infty$. Moreover, $s$ and $\psi(r)$ depend
neither on $\RR\succeq\RR^0$ nor on the parameterization $\Phi$.
\end{lemma}
\begin{proof}
We start with a given resolution $\RR^1$ of the surface $S$ and
adopt the notations above. Let $\ww= \yy \circ
\Phi=(w_1,\ldots,w_{n-1})$ and $\theta(z)=z^{m+1}\oth(z)$ with $\oth(0)\neq {\bf 0}$.
Let $\lambda(r) := (1+m)\oth(r)+
r\oth'(r)=r^{-m}\theta'(r)$ where the prime denotes the usual
derivative. It is an analytic mapping which does not vanish at
$r=0$. From the expression of $\beta$, we deduce
\begin{center}
\vspace{4pt}
$
\begin{array}{rll}
\vspace{4pt}
(\beta\circ\Phi)^*{\rd x_n} & = & Nr^{N-1}\rd r \\
\vspace{4pt}
(\beta\circ\Phi)^*{\rd x_j} & = & [r^{eN-1} (eNw_j+r(w_j)_r) +r^m \lambda_j] \rd r +
r^{eN}(w_j)_\vp\rd \vp.\\
\end{array}
$
\vspace{4pt}
\end{center}
Note that each $w_j$ is u-a-r-a and extends continuously on the
whole bottom circle $\CC$. But $(w_j)_r$ could even be unbounded.
However, by Remark \ref{rk:derivative-wrt-r}, each function
$r(w_j)_r$ is u-a-r-a and belongs to $\AAA_{>0}$. Taking this
property into account and since $\bg$ is the Euclidean metric,
we obtain
\vspace{4pt}
\begin{equation}\label{eq:coeffients-metric-bis}
\begin{array}{rll}
\vspace{4pt}
\wt{\bh} & = & (\beta\circ\Phi)^* (\rd x_1^2+\cdots+\rd x_n^2) \\
\vspace{4pt}
 & = & [N^2r^{2N-2} \; + \; r^{2m}\sum_j\lambda_j^2 \; + \;
 r^{eN+m-1}(\sum_j eN\lambda_jw_j \; +\cdots) \\
\vspace{4pt}
&  &  + \; r^{2eN-2}(\sum_je^2N^2w_j^2\; +\; \cdots)] \rd r^2 \; + \\
\vspace{4pt}
&  &  2[r^{eN+m}\sum_j\lambda_j(w_j)_\vp +  r^{2eN-1}(\sum_j eNw_j(w_j)_\vp  +  \cdots)] \rd r \rd\vp  \\
\vspace{4pt}
&  &  + \;[r^{2eN}\sum_j(w_j)_\vp^2] \rd\vp^2,
\end{array}
\end{equation}
where $\cdots$ stands for an element of $\AAA_{>0}$. Let $\psi(r) :=
r^{2(m-N)+2}\sum_j\lambda_j^2(r)$. Since $m\geq N$ and $e>1$ we can
define $\overline{A}_1,\overline{A}_2\in\AAA_{\geq 0}$ so that $A$,
the coefficient of $\rd r^2$ in (\ref{eq:coeffients-metric-bis})
writes as in (\ref{eq:coefficients_metric}). Notice that $\psi(r)$
does not depend on the resolution or the parameterization. \\
On the
other hand, the second summand of the coefficient of $\rd r\rd \vp$
in (\ref{eq:coeffients-metric-bis}) is given by
$r^{2eN-1}R_\vp+\overline{B}$ where $\overline{B}\in\AAA_{>2eN-1}$,
while the coefficient of $\rd\vp^2$ in
(\ref{eq:coeffients-metric-bis}) is just $r^{2eN}U$. \\
In order to complete the expressions of (\ref{eq:coefficients_metric}), let us
have a look at the first summand of the coefficient of $\rd r\rd
\vp$ in (\ref{eq:coeffients-metric-bis}). Consider the function
$$
h(\yy,z)=z^{eN+m}\sum_j\lambda_j(z)y_j,
$$
defined and analytic in a neighborhood of $E$ in $M\times\R$.
Applying Proposition~\ref{prop:(z,S)-relative-principal-expansion}
to $h$, there exists a resolution $\RR^0$ of the surface $S$ so
that, given any other resolution $\RR\succeq\RR^0$ and
any associated parametrization $\Phi$, the composition
$h^\Phi = h\circ\Phi$ on the open cylinder $\SSS^1\times]0,\eps]$
writes
\begin{equation}\label{eq:hphi}
h^\Phi(\vp,r)=P(r)+r^sH(\vp,r)
\end{equation}
where $s\in\Q_{\geq 0}$, $P(r)$ is a $\Q$-generalized polynomial and
$H$ is an u-a-r-a function that extends continuously to the bottom
circle $\CC$ such that $H|_\CC$ is either not constant if
$s<+\infty$ or, for $s=\infty$, $H$ is identically zero  and $P(r)$
is a convergent Puiseux series. The first summand of the coefficient
of $\rd r\rd \vp$ in (\ref{eq:coeffients-metric-bis}) is just the
partial derivative $(h^\Phi)_\vp$ (identically zero if $s=\infty$),
equal to $r^sH_\vp$ by (\ref{eq:hphi}). Since $H\in\AAA_{\geq 0}$,
we get Inequality (\ref{eq:estimates_s}). Moreover,
Proposition~\ref{prop:(z,S)-relative-principal-expansion} also
ensures that the exponent $s$ does not depend on the given
resolution $\RR$ (or on the parametrization $\Phi$) as long as it
dominates $\RR^0$. This completes the proof of the lemma.
\end{proof}

Now consider the function $f=f_0\circ\beta$ which is an analytic
function in a neighborhood of $E$ in $M\times\R$. We can assume that
$f_0(\oo)=0$ so that $f|_E\equiv 0$. Applying
Proposition~\ref{prop:(z,S)-relative-principal-expansion} to $f$,
there exists a resolution $\RR$ of $S$ and an associated
parametrization $\Phi$ such that expression (\ref{eq:expression-fR})
is valid: We can write either $f^\Phi=f^\Phi(r)$ as a convergent
Puiseux series only depending on $r$ or else
\begin{equation}\label{eq:fphi}
f^\Phi= a_0r^{\alpha_0} + \ldots + a_mr^{\alpha_m}+
r^{\alpha}F(\vp,r)=P(r)+r^{\alpha}F(\vp,r),
\end{equation}
where $a_j\in\R\setminus\{0\}$, $0 \leq
\alpha_0<\cdots<\alpha_m<\alpha$ are non-negative rational numbers
and $F$ is u-a-r-a and extends continuously to the whole cylinder
and the restriction $F|_\CC$ is not a constant function. Recall
moreover that
 $P(r)$, as well as the exponent
$\alpha$, are independent of the parameterization and of any
resolution $\RR'$ which dominates $\RR$. Therefore, we can suppose
that $\RR$ and $\Phi$ are chosen such that the expressions
(\ref{eq:coefficients_metric}) for the coefficients of the
transformed metric in Lemma~\ref{lm:expression_transf_metric} also
hold.

In what follows, we treat the degenerate case when $f^\Phi$ only
depends on $r$ as the case (\ref{eq:fphi}) with $\alpha$ as big as
we want but without requiring that $F|_C$ is not constant.

\smallskip
Up to a multiplication by a function that does not vanish on the
open cylinder, the differential equation provided by the transformed
vector field $\bar{\xi}$ writes:
\begin{equation}\label{eq:expression_xi}
\left\{
\begin{array}{rcl}\vspace{4pt}
\dot{r} & =& [P'(r)+(r^\alpha F)_r ] C\,\, -   \,\, r^\alpha F_\vp B\\
\vspace{4pt}
\dot{\vp} & = &   [P'(r)+(r^\alpha F)_r] B\,\, + r^\alpha F_\vp A
\vspace{4pt}
\end{array}
\right.
\end{equation}

We have several cases to deal with.

\medskip\noindent
{\em Case (1):} $\wt{\alpha}=2N-2+\alpha<\min\{s+\alpha_0-1,2eN+\alpha_0-2\}$. \\
From the expression of $\dot{\vp}$ in (\ref{eq:expression_xi}) we obtain
\begin{center}
\vspace{4pt}
$\dot{\vp}=r^{\wt{\alpha}}(N^2 F_\vp+\Delta), \mbox{
whith } \Delta\in\AAA_{>0}. $
\vspace{4pt}
\end{center}
On the other hand, from
\begin{center}
\vspace{4pt}
$ 2eN+\alpha-1\geq 2eN+\alpha_0-1>\wt{\alpha}+1\;\;$ and
$\;\;s+\alpha\geq s+\alpha_0>\wt{\alpha}+1$
\vspace{4pt}
\end{center}
we deduce $\dot{r}\in\AAA_{>\wt{\alpha}+1}$. Eventually $\bar{\xi}$
is in the non-monodromic case (\ref{eq:nonmon}) of Proposition
\ref{pro:dic-or-nonmon}.

\medskip\noindent
{\em Case (2):} $\wt{\alpha}=2N-2+\alpha\geq\min\{s+\alpha_0-1,2eN+\alpha_0-2\}$. \\
Using (\ref{eq:estimates_s}) in this case, we find $\alpha>\alpha_0$ and
thus $P\neq 0$ and $a_0\alpha_0\neq 0$.
We distinguish two sub-cases:

\smallskip\noindent
{\em Case (2a):} $2eN+\alpha_0-2<s+\alpha_0-1$. \\
 We deduce first
that $\dot{\vp}=r^{2eN+\alpha_0-2}(G_\vp+\Delta)$ where
\begin{equation}\label{eq:dotphi}
G=\left\{
\begin{array}{ll}
\vspace{4pt}
\alpha_0 a_0 R , & \hbox{ if $2eN+\alpha_0-2<\wt{\alpha}$;} \\
\vspace{4pt}
\alpha_0 a_0 R+N^2F & \hbox{ if $2eN+\alpha_0-2=\wt{\alpha}$.}
\end{array}
\right.
\end{equation}
and $\Delta\in\AAA_{>0}$. We observe that the function
$r^{2eN+\alpha_0-2}G$ is of the form $g^\Phi=g\circ\Phi$ for some
ramified-analytic function $g$ on a neighborhood of $E$ in
$M\times\R_{\geq 0}$. Thus the function $G$ is continuous on the
cylinder and u-a-r-a. \\
On the other hand, the term $r^\alpha F_\vp\, B$ in the
expression for $\dot{r}$ in (\ref{eq:expression_xi}) belongs to
$\AAA_{>2eN+\alpha_0-1}$. Thus
$\dot{r}=r^{2eN+\alpha_0-1}(a_0\alpha_0 U+\Upsilon)$ for
$\Upsilon\in\AAA_{>0}$.\\ If $G|_{\CC}$ is not constant, our situation is
non-monodromic in the sense of Proposition \ref{pro:dic-or-nonmon}. Otherwise,
 thanks to Lemma~\ref{lm:U}, it is the dicritical case of Equation (\ref{eq:dic})
with $\mu=2eN+\alpha_0-1$.

\medskip\noindent
{\em Case (2b):} $s+\alpha_0-1\leq 2eN+\alpha_0-2$. \\
This case is the most difficult since several of the terms
involved in the expression of $\dot{\vp}$ may be of the same order
with respect to $r$ so that all the ``initial parts" which are
derivatives with respect to $\vp$ of a function may cancel.

From Equation (\ref{eq:estimates_s}) and $m\geq N$,
we first find
\vspace{4pt}
\begin{equation}\label{eq:estimate-alpha}
\alpha+2eN-1\geq \alpha+s\geq 2eN+\alpha_0+1.
\vspace{4pt}
\end{equation}
Using Equations (\ref{eq:coefficients_metric}) and (\ref{eq:estimate-alpha})
we get $r^\alpha F_\vp B\in\AAA_{\geq
2eN+\alpha_0+1}$. Thus, in (\ref{eq:expression_xi}), we obtain
\begin{equation}\label{eq:dot_r_final}
\vspace{4pt} \dot{r}=r^{2eN+\alpha_0-1}(a_0\alpha_0
U+\Upsilon),\;\;\mbox{ for }\Upsilon\in\AAA_{>0}.
\end{equation}
Using again (\ref{eq:estimate-alpha}) we find the following estimates on
the order of some terms in the expression of
$\dot{\vp}$:
\vspace{4pt}
\begin{equation*}
\begin{aligned}
\vspace{4pt}
P'(r)\overline{B} & \in\AAA_{>2eN+\alpha_0-2},\\
\vspace{4pt}
(r^\alpha F_\vp) r^{eN+m-1}\overline{A}_1 & \in\AAA_{>2eN+\alpha_0},\\
\vspace{4pt}
(r^\alpha F_\vp) r^{2eN-2}\overline{A}_2 & \in\AAA_{\geq 2eN+\alpha_0} \\
\vspace{4pt}
(r^\alpha F)_\vp B & \in\AAA_{\geq 2eN+\alpha_0}.
\end{aligned}
\end{equation*}

This allow us to write $\dot{\vp}$ as
\begin{equation}\label{eq:dot-phi-final}
    \dot{\vp}=-P'(r)[r^s H_\vp
    +r^{2eN-1}R_\vp]+r^{\wt{\alpha}}(N^2+\psi(r))F_\vp
    +\Delta=G_\vp+\Delta,
\end{equation}
where $\Delta\in\AAA_{>2eN+\alpha_0-2}$ and
$$
G=-P'(r)[r^s H
    +r^{2eN-1}R]+r^{\wt{\alpha}}(N^2+\psi(r))F.
$$
Once again $G=g^\Phi=g\circ\Phi$ for some function $g$ in a
neighborhood of $E$ in $M\times\R_{\geq 0}$ which is ramified
analytic along $E$. From the definition of $R$ and
Remark~\ref{rm:about-F} for $H$ and $F$, the function $g$ depends
only on $f$ and $\beta$ but not on the resolution $\RR$ or on an associated
parameterization $\Phi$. So, up to further finitely
many blowing-ups and applying
Proposition~\ref{prop:(z,S)-relative-principal-expansion}, we can
assume that
$$
G(\vp,r)=Q(r)+r^\rho\wt{G}(\vp,r)
$$
where $Q$ is a $\Q$-generalized polynomial, $\rho\in\Q_{>0}$ and
$\wt{G}$ is an u-a-r-a function which extends continuously to the
bottom circle $\CC$ with, either $\wt{G}\mid_\CC$ is not constant or
$\rho$ can be chosen as large as we want (we just need $\rho>2eN+\alpha_0-2$). \\
Two cases are to be considered:\\
- If $\rho\leq 2eN+\alpha_0-2$, Equation (\ref{eq:dot-phi-final})
writes $\dot{\vp}=r^\rho(\wt{G}_\vp+\wt{\Delta})$ for
$\wt{\Delta}\in\AAA_{>0}$. Combined with (\ref{eq:dot_r_final}), we
find a non-monodromic situation (\ref{eq:nonmon}). \\
- If $\rho>2eN+\alpha_0-2$ then we are in the dicritical situation (\ref{eq:dic})
with $\mu=2eN+\alpha_0-1$ thanks to Lemma~\ref{lm:U}.

\vspace{5pt} This finishes the proof of the main theorem in the CTC
case.
\bigskip
\noindent{\bf Open tangent cone case.}
\\
Let us have a quick look at the open tangent case (OTC).

We first choose orthonormal coordinates $x =(x_1,\ldots,x_n)$.
Consider the opening blow-up mapping
$\beta:M\times\R_{\geq 0} \to \R^n$,
 $(\yy,z) \to z\yy$, where $M = \SSS^{n-1}$, as in
(\ref{eq:map-beta}), and let $S=\beta^{-1}(S_0)$. Let $\RR =
(\wt{S},\wt{E},\sigma)$ be a $(S,z)$-resolution of $S$ as in
Theorem~\ref{th:reduction-singularities} and pick an associated
parameterization $\Phi: \SSS^1\times[0,\eps] \to S$ satisfying the
conditions of Proposition~\ref{pro:parameterization-of-surfaces}.

Let $\wt{\bh}=(\beta\circ\Phi)^\ast\bh$ be the pull-back metric in
the open cylinder. With computations similar to those done in the
proof of Lemma \ref{lm:expression_transf_metric}, and using
Remark~\ref{rk:derivative-wrt-r}, we can write
\begin{center}
\vspace{4pt} $\wt{\bh}=(1+ \overline{A})\rd r^2 + 2r\overline{B}\rd r\rd \vp + r^2U\rd \vp^2$,
\vspace{4pt}
\end{center}
where $U=\sum_{i}(w_i)_\vp^2$, $\overline{A} =r^2\sum_{i}(w_i)_r^2$ and
$\overline{B} = r\sum_i(w_i)_r(w_i)_\vp$, since $\sum_i w_i^2=1$.
We note that $\overline{A},\overline{B}\in \AAA_{>0}$.

Writing $f^\Phi = f_0\circ\beta\circ\Phi$, the pull-back
$\bar{\xi}=(\beta\circ\Phi)^\ast\nbh f_0$ of the restricted
gradient vector field has the following associated system of
differential equations (up to the multiplication by the determinant of
the metric $\wt{\bh}$):
\vspace{4pt}
\begin{equation}\label{eq:differential-eq-xitilda-OTC}
\vspace{4pt}
\left\{
\begin{array}{l}
\vspace{4pt}
\dot{r}=r^2U(f^\Phi)_r - r\overline{B}(f^\Phi)_\vp\\
\vspace{4pt}
\dot{\vp}=-r\overline{B}(f^\Phi)_r+(1+\overline{A})(f^\Phi)_\vp.
\end{array}
\right.
\end{equation}

We consider cases $(a)$ or $(b)$ of Proposition~\ref{prop:(z,S)-relative-principal-expansion}
for the function $f:=f_0\circ\beta$.

\medskip
In case $(a)$ the function $f$ depends only on $z$ and thus
$(f^\Phi)_\vp\equiv 0$. Dividing
(\ref{eq:differential-eq-xitilda-OTC}) by $r(f^\Phi)_r$, which does
not vanish on the open cylinder, we obtain the dicritical situation
of Proposition~\ref{pro:dic-or-nonmon}.

\medskip
In case $(b)$, we assume that the resolution $\RR$ is such that
$$
f^\Phi(\vp,r)=P(r)+r^\alpha F(\vp,r),
$$
where $P(r)=a_0r^{\alpha_0}+\cdots+a_mr^{\alpha_m}$,
$\alpha_m<\alpha$ if $a_0 \neq 0$, and $F$ extends continuously to
the bottom circle $\CC$ and its restriction $F|_\CC$ is not
constant.

\noindent If $\alpha\leq 1$, Equations
(\ref{eq:differential-eq-xitilda-OTC}) become
\begin{center}
\vspace{4pt} $\dot{r} \in \AAA_{\geq\alpha+1}$ and
$\dot{\vp}=r^\alpha[F_\vp + \Delta], $ \vspace{4pt}
\end{center}
where $\Delta \in \AAA_{>0}$. We have a non-monodromic situation as
in (\ref{eq:nonmon}) and we are done.

\noindent If $\alpha>1$, we have two sub-cases:

\noindent - Case $\alpha=\alpha_0$. This means that $P\equiv 0$. Thus
$$
r(f^\Phi)_r=r^\alpha(\alpha F+rF_r) \mbox{ and
}(f^\Phi)_\vp=r^\alpha F_\vp.
$$
Using Remark~\ref{rk:derivative-wrt-r} and Equation (\ref{eq:differential-eq-xitilda-OTC}),
we find $\dot{\vp}=r^\alpha(F_\vp+\Delta)$ where $\Delta\in\AAA_{>0}$. We still have
$\dot{r}\in\AAA_{\geq\alpha+1}$ and thus we obtain a
non-monodromic situation.

\noindent Case $\alpha>\alpha_0$. We deduce
\begin{center}
\vspace{4pt}
$\dot{r}=r^{\alpha_0+1}(a_0\alpha_0
U+\Upsilon)\;\;$ and $\;\;\dot{\vp}\in\AAA_{>\alpha_0}
$
\vspace{4pt}
\end{center}
with $\Upsilon\in\AAA_{>0}$. We obtain the dicritical situation
(\ref{eq:dic}) with $\mu=\alpha_0+1$ thanks to Lemma~\ref{lm:U}.

\medskip
 This finishes all the cases and the proof of the Main
Theorem~\ref{thm:MainResult}.
\section{Consequences}\label{section:consequences}
Now we prove Corollary \ref{cor:Pfaffian} and Theorem~\ref{cor:formal-separatrix} 
as consequences of our  main result, Theorem \ref{thm:MainResult}. We also sketch 
a proof of the more elementary Proposition~\ref{pro:trajectory-to-origin}.

\medskip
\noindent {\em Proof of Corollary~\ref{cor:Pfaffian}.} Suppose that
$\absgamma\subset S_0$, a connected component of
$X\setminus\{\oo\}$. Let $\beta$ be an opening blowing-up of $S_0$
and let $\RR = (\wt{S},\wt{E},\sigma)$ be a resolution of the
surface $S=\beta^{-1}(S_0)$.  Theorem \ref{thm:MainResult} ensures
that the lifting  $\LL = (\beta\circ\sigma)^{-1}(\absgamma)$
accumulates at a single point $p$ of $\wt{E}$. Thus $\LL$ is
contained in a simply connected semi-analytic open set of the strict
transform $S'=\sigma^{-1}(S)$ where the foliation $\mathcal{F}$ has
no singularities.
Using Haefliger's theorem (\cite{Hae,Mou-R,Sp}), we deduce that
$\LL$ is a Rolle's leaf of $\mathcal{F}$ and thus a pfaffian set.
Its image $\absgamma = \beta(\sigma(\LL))$ is a sub-pfaffian set in
$\R^n$ since $\sigma$ and $\beta$ are proper mappings.
\hfill{$\square$}

\bigskip
\noindent 
{\em Proof of Proposition \ref{pro:trajectory-to-origin}.}
 Let $S_0$ be a connected component of $X\setminus\{\oo\}$,
homeomorphic to the semi-open cylinder $\SSS^1\times]0,\eps]$ for
$\eps$ small. Denote by $\CC_\eps$ the image of
$\SSS^1\times\{\eps\}$ by such homeomorphism. We consider two cases.

\medskip
\noindent
{\bf Case 1}: The function $f_0|_{S_0}$ is negative and has no critical point.\\
Let $a_0<0$ be the minimum of the function $f_0$ restricted to
$\CC_\eps$. Consider a point $p\in S_0$ for which $a_0<f_0(p)<0$ and
let $\gamma_p$ be the trajectory of the restricted gradient
vector field $\nabla_{\bh}f_0$ starting at $p$. Since
$t\to f_0(\gamma_p(t))$ is increasing, $\gamma_p$ is defined for all
positive $t$ and $\lim_{t\to\infty}\gamma_p(t)=\oo$.

\medskip\noindent
{\bf Case 2.} Suppose $f_0^{-1}(0)\cap S_0 \neq \emptyset$.\\
Up to taking $-f_0^2$ instead of $f_0$, we assume that $f_0 \leq 0$
on $\clos(S_0)$ and that $Z_0 = f_0^{-1}(0) \cap \clos (S_0) (=
\crit (f_0\mid_{S_0}))$ intersects with $S_0$.\\
Let $U$ be a connected component of $S_0\setminus Z_0$. Since $S_0$
is topologically a cylinder and $Z_0$ consist of finitely many
analytic half-branches at $\oo$ (up to taking $\eps$ smaller), the component 
$U$ is simply connected. In fact, we can take a triangle
$\Sigma$ in the plane with sides $\sigma_1,\sigma_2,\sigma_3$ and a
continuous map $\kappa:\Sigma\to\clos(U)$ restricting to a
diffeomorphism between $\Sigma\setminus (\sigma_1\cup\sigma_2)$ and
$U$, sending each of the sides $\sigma_1$ or $\sigma_2$
homeomorphically to a half-branch of $Z_0$ and sending the side
$\sigma_3$ onto $\clos(U)\cap\CC_\eps$. Note that, if there are at least two
half-branches of $Z_0$ then $\kappa$ is a homeomorphism, otherwise
$\clos(U)=\clos(S_0)$ and $\kappa$ is just a quotient map gluing the
two sides $\sigma_1,\sigma_2$ together.

\smallskip
Since $\clos (U)$ is invariant, we can carry $\nabla_{\bh}f_0$ onto 
$\Sigma$ via $\kappa$ (which is singular along
$\sigma_1\cap\sigma_2$). It will be denoted by $\chi_0$ while we
will denote $g_0=\kappa^* f_0$ and
$\bv=\sigma_1\cap\sigma_2=\kappa^{-1}(\oo)$. We just have to prove that
there exists a trajectory of $\chi_0$ accumulating to $\bv$.

\smallskip
We use the following properties:
\begin{enumerate}
\item Up to taking a smaller $\eps$, each point
$x \in \sigma_1\cup\sigma_2\setminus\{\bv\}$ is the accumulation point
of a unique trajectory of $\chi_0$.
\item No non-stationary trajectory of $\chi_0$ can have its $\alpha$-limit 
point and its $\omega$-limit point both in $\sigma_1\cup\sigma_2\setminus\{\bv\}$.
\end{enumerate}
The first property is easy to prove using local coordinates or using
the classical \L ojasiewicz's retraction map of the gradient (cf.
\cite{Loj}). The second one is a consequence of the fact that
$g_0(\sigma_1\cup\sigma_2)=0$ and that $g_0$ increases along
trajectories of $\chi_0$.

\smallskip\noindent
{\bf Claim.} \em There exists $t_\eps <0$ such that for $t\in
]t_\eps,0[$, the fiber $g_0^{-1} (t)\subset \Sigma$ is connected.\em
\\
\em Proof of the Claim. \em Each connected component of a (non-empty) 
fiber $g_0^{-1}(t)$
with $t<0$ is either homeomorphic to a circle or to a closed segment
with extremities on $\sigma_3$. Since $g_0$ has no critical points in the 
interior of the triangle $\Sigma$, the first case cannot occur. 
On the other hand, the restriction $g_0|_{\sigma_3}$ vanishes only 
at the extremities. If we take $t_\eps$ equal to the maximum of the critical 
values of this restriction, $g_0$ takes any value $t\in]t_\eps,0[$ 
exactly twice along $\sigma_3$ and this proves the claim.
\hfill{$\square$}

\smallskip\noindent
For $i=1,2$, choose $x_i\in\sigma_i\setminus\{v\}$ and let $\g_i$ be
the trajectory of $\chi_0$ accumulating to $x_i$. Take $t_0$ with
$t_\eps<t_0<0$ such that $\g_i$ cuts the fiber $g_0^{-1}(t_0)$,
necessarily in a single point $y_i$. Let $I$ be the closed segment
in $g_0^{-1}(t_0)$ joining $y_1$ and $y_2$. Consider the domain $\Lambda$
in $\Sigma$ enclosed by the piecewise smooth
closed curve formed by the segments $[x_i,\bv]$ in $\sigma_i$,
$[y_i,x_i]$ in $\g_i$ and $I$. By construction, $\chi_0$ enters $\Lambda$
only through the segment $I$ and leaves $\Lambda$ positively invariant.

For each $z$ in one of the semi-open sides $[x_1,\bv[$ or $[x_2,\bv[$, 
thanks to properties (1) and (2) above,  there exists a unique point
$\tau(z)\in I$ such that the trajectory starting at $\tau(z)$
accumulates to $z$ for positive infinite time. Moreover, orienting positively
$I$ from $y_1$ to $y_2$, we find that
$\tau(z)<\tau(w)$ whenever $z\in[x_1,\bv[$ and $w\in[x_2,\bv[$, or
$z \in [x_1,\bv[$ and $w \in ]z,\bv[$, or $w\in[x_2,\bv[$ and $z \in ]w,\bv[$.

Let $a=\sup\{\tau(z)/z\in[x_1,\bv[\}$ and $b=\inf\{\tau(z)/z\in[x_2,\bv[\}$. 
Thus $a\leq b$ and for every point $y\in[a,b]$ in the segment $I$, the trajectory of 
$\chi_0$ starting at $y$ accumulates to $\bv$. 
\hfill{$\square$}

\bigskip
\noindent {\em Proof of Theorem \ref{cor:formal-separatrix}.} We
begin with the definition of {\em formal asymptotic expansion}. A
formal curve $\widehat{\Gamma}$ at the origin of $\R^n$ is a formal
Puiseux parameterization $\wht{\Gamma}(T)=(\wht{\Gamma}_1(T),\ldots,
\wht{\Gamma}_{n-1}(T),T^N)$, where each $\wht{\Gamma}_i$ is a formal
power series in the single indeterminate $T$ with no constant term. A
trajectory $\gamma$ has an asymptotic expansion $\wht{\Gamma}$ at
the origin if it can be smoothly parameterized as
\begin{center}
\vspace{4pt}
$
z\to \gamma(z)=(\gamma_1(z),\ldots,\gamma_{n-1}(z),z^N), z>0
$
\vspace{4pt}
\end{center}
and the component $\gamma_i$ admits $\wht{\Gamma}_i$ as expansion.

\smallskip
\noindent
If the critical locus
of $f_0|_{S_0}$ is not empty then, each connected component of this locus in a semi-analytic
invariant set the restricted gradient, so point (ii) is true.

\smallskip
\noindent
We assume the restricted gradient of $f_0$ does not vanish in $S_0$. \\
Since any restricted gradient trajectory $\gamma$ is not oscillating
at $\oo$, the function $z \circ \gamma (t)$ decreases strictly to $0$
as $t\to +\infty$. Thus it admits a continuous parameterization $z\to \gamma(z)
=(\gamma_1(z),\ldots, \gamma_{n-1}(z),z)$, for $z\geq 0$, which is analytic for $z>0$.

\smallskip
Let $\beta$ be an opening blowing-up of $S_0$ and $S=\beta^{-1}(S_0)$.
Let $\RR=(\wt{S},\wt{E},\sigma)$ be a resolution of $S$ and let
$\RR'=(S',E',\sigma')$ be the strict resolution associated to $\RR$
as in Theorem~\ref{th:reduction-singularities}. Let
$\wt{\bh}=(\beta\circ\sigma)^*{\bf g}$ and
$\wt{f}=f_0\circ\beta\circ\sigma$. The metric $\wt{{\bf h}}$
degenerates along the divisor $\wt{E}$, so the gradient vector field
$\nabla_{\wt{\bh}}\wt{f}$ is defined only on
$\wt{S}\setminus\wt{E}$. However, we can define a one-dimensional
analytic foliation $\wt{\FF}$ in $\wt{S}$ whose singular set $\sing(\wt{\FF})$
is a finite subset of $\wt{E}$ and
such that $\nabla_{\wt{{\bf h}}}\wt{f}$ is a local generator of
$\wt{\FF}$ at any point of $\wt{S}\setminus\wt{E}$.

The reduction of singularities of an analytic foliation on a smooth surface
(\cite{Sei}) and the compactness
of $\wt{E}$ ensure we can assume that any singularity $p\in \sing(\wt{\FF})$
is simple: a local generator $\xi_p$ has a non-nilpotent linear part at $p$
with eigenvalues $\lambda$ and $\mu\neq 0$.

Let $\Sigma:=E'\cap \sing(\wt{E})$ be the finite set of singular points
of the strict divisor $E'$.

\smallskip\noindent
In order to complete the proof, we will show that only situations
(1) or (2) below happen and the result holds true in both cases.

\smallskip\noindent
(1) {\em Dicritical situation:} There is a point $p\in E'\setminus
\Sigma$, not singular for $\wt{\FF}$, and such that $E'$ is transverse
to $\wt{\FF}$ at $p$. The leave $\LL_p$ of $\wt{\FF}$ through $p$ is
a non-singular analytic curve and transverse to $E'$ at $p$. The
image $\beta\circ\sigma(\LL_p\cap S')$ is an analytic separatrix for
the restricted gradient on $S_0$ accumulating to the origin. In fact
through each point $q\in E'$ in a neighborhood of $p$, there is a
unique analytic separatrix through $q$.

\smallskip\noindent
(2) {\em A non corner singularity:} The strict divisor $E'$ is an
invariant set of $\wt{\FF}$ and there is a point $p\in (E'\setminus
\Sigma)\cap \sing(\wt{\FF})$. A local generator $\xi_p$ of
$\wt{\FF}$ at $p$ has a linear part with two real eigenvalues. One
eigen-direction is tangent to $E'$ and the other one is transverse
to $E'$. The theory of local invariant manifolds (see for instance
\cite{Hir-P-S}) provides a formal invariant non-singular manifold
$\widehat{W}$ at $p$  which is tangent to the transverse
eigen-direction\footnote{If the corresponding eigenvalue is non-zero,
Briot-Bouquet's theorem guarantees the convergence of $\widehat{W}$}.
We also get a $C^\infty$ invariant
manifold $W$ through $p$ having $\widehat{W}$ as asymptotic
expansion at $p$. The image $\beta\circ\sigma(W\cap S')$ is the
desired characteristic trajectory $\gamma$ of the restricted
gradient.

\medskip
\noindent
In order to find a contradiction, we assume that neither case (1) or (2) above holds.
Thus $E'$ is invariant for $\wt{\FF}$ and $\sing(\wt{\FF})\cap E'=\Sigma$.
\\
Since any point $p\in\Sigma$ is a simple
singularity, the two components of $\wt{E}$ at $p$ are the two local
analytic separatrices of $\wt{\FF}$ at $p$. \\
Let $\{Q_p^j\}_{j=1,2,3,4}$ be the open ``quadrants'' of
$\UU_p\setminus \wt{E}$ in a small coordinate neighborhood $\UU_p$
of $p$ in $\wt{S}$. Let $J(p)\subset\{1,2,3,4\}$ be the subset of
$j$ for which $Q_p^j\subset S'$ (see part (iv) of
Theorem~\ref{th:reduction-singularities}). For $j\in J(p)$, the
quadrant $Q_p^j$ is either: \\
-  of {\em saddle type}, if any trajectory of the restricted
gradient $\nabla_{\wt{\bf h}}\wt{f}$ through a point in $Q_p^j$
escapes from $Q_p^j$ for positive and negative time; \\
-  of {\em node-source type}, if any trajectory escapes for positive
time but accumulates to $p$ for negative time; \\
-  of {\em node-sink type}, if each trajectory escapes for negative
time but accumulates to $p$ for positive time.

\smallskip\noindent
We have two possibilities:

(a) Each quadrant $Q_p^j$ is of
saddle-type for all $p\in \Sigma$ and all $j\in J(p)$ or there are no
singularities at all ($\Sigma=\emptyset$). In this case, classical
arguments show that the dynamics is ``monodromic'' in a neighborhood
of $E'$ in $S'$: there exist an analytic half-branch $\Lambda$
through a point $q\in E'\setminus \Sigma$ contained in $S'$, transverse
to $\wt{\FF}$, a neighborhood $\Lambda_0$ of $q$ in $\Lambda$ and a
{\em Poincar\'{e} first return map} $P:\Lambda_0\to\Lambda$ such that
for $x\in\Lambda_0$ given, the leaf $\LL_x$ through $x$ cuts again
$\Lambda$ at $P(x)$ after visiting all the quadrants $Q_p^j$.
Since a gradient vector field cannot have closed orbits, $P$ has no
fixed points. Poincar\'{e}-Bendixson's type arguments imply
there are leaves of $\wt{\FF}$ in $S'$ accumulating to
the whole divisor $E'$, thus producing spiraling
trajectories of the restricted gradient, which contradicts
Theorem~\ref{thm:MainResult}.

(b) There is a quadrant $Q^{j_0}_{p_0}\subset S'$ of node type for
some singularity $p_0\in \Sigma$ and some $j_0\in J(p_0)$. Suppose for
instance that it is of node-source type (the case of node-sink type
is analogous in reversing time). Consider one of the local analytic
separatrices of $\wt{\FF}$ at $p_0$. There is a connected component
of $E'\setminus \Sigma$, say $\EE_1$, which meets such a separatrix. Since
$E'$ is invariant, $\EE_1$ is a leaf of $\wt{\FF}$. Let $p_1\in \Sigma$ be
the other accumulation point of $\EE_1$, different from $p_0$. The
flow-box theorem shows that there is a point $q_0\in Q^{j_0}_{p_0}$
such that the trajectory $\gamma_0$ issued from the point $q_0$
visits a point in a quadrant $Q^{j_1}_{p_1}$ for some $j_1\in
J(p_1)$.
\\
By definition of a node-source type, the quadrant $Q^{j_1}_{p_1}$ cannot be of node-source
type. If $Q^{j_1}_{p_1}$ is of saddle-type, we consider the connected component
$\EE_2$ of $E'\setminus \Sigma$ meeting the local analytic separatrix at
$p_1$ which is not contained in $\EE_1$ and the point $p_2\in \Sigma$ such
that $\clos(\EE_2)\setminus \EE_2=\{p_1,p_2\}$. In this case, choosing
$q_0$ in the initial quadrant $Q^{j_0}_{p_0}$ sufficiently close to
$p_0$, we can suppose that the trajectory $\gamma_0$ also visits
some quadrant $Q^{j_2}_{p_2}$ for some $j_2\in J(p_2)$.
\\
Continuing this way, if all the visited quadrants are of saddle-type
we construct a sequence of singularities $p_1,p_2,\ldots$ different
from $p_0$. Since $\Sigma$ is finite, we create a cycle
$p_l,p_{l+1},\ldots,p_m=p_l$, with $l$ minimum for this property.
Since $\EE_{m}$ is not equal to $\EE_{l}$ (otherwise
$p_{m-1}=p_{l-1}$ against the minimality of $l$) we find three
local analytic separatrices through $p_l$, say $\EE_l, \EE_{l+1}, \EE_m$,
which is a contradiction with the fact that $p_l$ is a simple
singularity.

Thus, there exist $p_k\in \Sigma$, for some $k\geq 1$ and a quadrant $Q^{j_k}_{p_k}$
of node-sink type, for some $j_k\in J(p_k)$, which intersects the trajectory
$\gamma_0$. Then $p_0$ is the $\alpha$-limit point of $\gamma_0$ and
$p_k$ a $\omega$-limit point.
Its image $\beta\circ\sigma\circ\gamma_0$ is
a trajectory of the restricted gradient for which the origin is the $\alpha$
and the $\omega$ limit point which is impossible since the function $f_0$
increases strictly along this trajectory.
\section{Thanks}
This work started during the Thematic Program at the Fields Institute on \em O-minimal
Structures and Real Analytic Geometry January-June 2009. \em  The authors are both very grateful
to the Fields Institute and its staff for support, facilities and very good working conditions
they found there.

The first author would like to thank the Departamento de \'Algebra, Geometr\'{\i}a y
Topolog\'{\i}a of the University of Valladolid for the support and the working conditions
provided while visiting to complete this work.

The second author was partially supported by the research projects MTM2007-66262 (Ministerio de
Educaci\'{o}n y Ciencia) and VA059A07 (Junta de Castilla y Le\'{o}n) and by {\em Plan Nacional de
Movilidad de RR.HH. 2008/11, Modalidad "Jos\'{e} Castillejo"}.

The authors want to thank also O. LeGal for useful commentaries and
remarks.

\end{document}